\documentclass[english]{article}
\usepackage{ae,aecompl}
\usepackage[T1]{fontenc}
\usepackage[latin9]{inputenc}
\usepackage[letterpaper]{geometry}
\usepackage{fancyhdr}
\usepackage{babel}
\usepackage{array}
\usepackage{longtable}
\usepackage{amsmath}
\usepackage{amsthm}
\usepackage{amssymb}
\usepackage[numbers]{natbib}
\usepackage[unicode=true,
 bookmarks=true,bookmarksnumbered=false,bookmarksopen=true,bookmarksopenlevel=2,
 breaklinks=false,pdfborder={0 0 1},backref=false,colorlinks=false]
 {hyperref}

\makeatletter

\newcommand{\noun}[1]{\textsc{#1}}
\providecommand{\tabularnewline}{\\}

\theoremstyle{plain}
\newtheorem{thm}{\protect\theoremname}
  \theoremstyle{definition}
  \newtheorem{defn}[thm]{\protect\definitionname}
  \theoremstyle{definition}
  \newtheorem{example}[thm]{\protect\examplename}
  \theoremstyle{plain}
  \newtheorem{prop}[thm]{\protect\propositionname}
  \theoremstyle{plain}
  \newtheorem{lem}[thm]{\protect\lemmaname}
  \theoremstyle{definition}
  \newtheorem{condition}[thm]{\protect\conditionname}
  \theoremstyle{plain}
  \newtheorem{lyxalgorithm}[thm]{\protect\algorithmname}

\@ifundefined{date}{}{\date{}}

\usepackage{tabstackengine}

\makeatother

  \providecommand{\algorithmname}{Algorithm}
  \providecommand{\conditionname}{Condition}
  \providecommand{\definitionname}{Definition}
  \providecommand{\examplename}{Example}
  \providecommand{\lemmaname}{Lemma}
  \providecommand{\propositionname}{Proposition}
\providecommand{\theoremname}{Theorem}

\begin{document}

\lhead{Projective Equivalence for the Roots of Unity}

\rhead{Hang Fu}

\title{\noun{Projective Equivalence for the Roots of Unity}}

\author{\noun{Hang Fu}}
\maketitle
\begin{quote}
\textbf{\small{}Abstract.}{\small{} Let $\mu_{\infty}\subseteq\mathbb{C}$
be the collection of roots of unity and $\mathcal{C}_{n}:=\{(s_{1},\cdots,s_{n})\in\mu_{\infty}^{n}:s_{i}\neq s_{j}\text{ for any }1\leq i<j\leq n\}$.
Two elements $(s_{1},\cdots,s_{n})$ and $(t_{1},\cdots,t_{n})$ of
$\mathcal{C}_{n}$ are said to be projectively equivalent if there
exists $\gamma\in\textup{PGL}(2,\mathbb{C})$ such that $\gamma(s_{i})=t_{i}$
for any $1\leq i\leq n$. In this article, we will give a complete
classification for the projectively equivalent pairs. As a consequence,
we will show that the maximal length for the nontrivial projectively
equivalent pairs is $14$.}{\small \par}

\textbf{\small{}Keywords.}{\small{} Roots of unity $\cdot$ Diophantine
equations $\cdot$ Unlikely intersections}{\small \par}

\textbf{\small{}Mathematics Subject Classification.}{\small{} 11Y50
$\cdot$ 11L03 $\cdot$ 11D72}{\small \par}
\end{quote}

\section{\label{section1}Introduction and the statements of main results}

Let $\mu_{\infty}\subseteq\mathbb{C}$ be the collection of roots
of unity and $\textup{Arg}/(2\pi):\mu_{\infty}\to\mathbb{Q}/\mathbb{Z}$
the canonical isomorphism.
\begin{defn}
\label{definition1}Let
\[
\mathcal{C}_{n}:=\{(s_{1},\cdots,s_{n})\in\mu_{\infty}^{n}:s_{i}\neq s_{j}\text{ for any }1\leq i<j\leq n\}.
\]
Two elements $(s_{1},\cdots,s_{n})$ and $(t_{1},\cdots,t_{n})$ of
$\mathcal{C}_{n}$ are said to be projectively equivalent if there
exists $\gamma\in\textup{PGL}(2,\mathbb{C})$ such that $\gamma(s_{i})=t_{i}$
for any $1\leq i\leq n$. We write $(s_{1},\cdots,s_{n})\sim(t_{1},\cdots,t_{n})$
to denote that they are projectively equivalent. Let $\mathcal{C}_{n}^{+}$
be the isomorphic image of $\mathcal{C}_{n}$ under $(\textup{Arg}/(2\pi))^{n}$.
The corresponding equivalence relation in $\mathcal{C}_{n}^{+}$ is
also denoted by $\sim$.
\end{defn}

The following examples are immediate.
\begin{example}
\label{example2}$ $
\begin{itemize}
\item Rotation: $(s_{1},\cdots,s_{n})\sim(ss_{1},\cdots,ss_{n})$ for any
$(s_{1},\cdots,s_{n})\in\mathcal{C}_{n}$ and $s\in\mu_{\infty}$.
\item Inversion: $(s_{1},\cdots,s_{n})\sim(1/s_{1},\cdots,1/s_{n})$ for
any $(s_{1},\cdots,s_{n})\in\mathcal{C}_{n}$.
\item If $n\leq3$, then $(s_{1},\cdots,s_{n})\sim(t_{1},\cdots,t_{n})$
for any $(s_{1},\cdots,s_{n}),(t_{1},\cdots,t_{n})\in\mathcal{C}_{n}$.
\end{itemize}
\end{example}

Our objective in this article is to describe $\mathcal{C}_{n}/\sim$
for general $n$. In particular, it is reasonable to expect that the
nontrivial projectively equivalent pairs cannot be arbitrarily long.
While the existence of such an upper bound can be implied by a result
of Schlickewei \citep{MR1393508}, our next result shows that the
maximal length is $14$.
\begin{thm}
\label{theorem3}Up to rotations, inversions, permutations, and Galois
conjugations, there are two nontrivial projectively equivalent pairs
in $\mathcal{C}_{14}^{+}$:
\begin{eqnarray*}
 &  & (0,1/30,1/15,1/10,2/15,1/6,1/5,3/10,11/30,7/15,3/5,11/15,5/6,9/10)\\
 & \sim & (0,1/15,1/6,3/10,13/30,8/15,3/5,7/10,11/15,23/30,4/5,5/6,13/15,9/10)
\end{eqnarray*}
and
\begin{eqnarray*}
 &  & (0,1/60,1/30,1/15,1/10,2/15,1/5,11/30,31/60,2/3,5/6,9/10,14/15,29/30)\\
 & \sim & (0,3/20,3/10,7/15,8/15,17/30,3/5,19/30,13/20,2/3,7/10,11/15,23/30,5/6).
\end{eqnarray*}
There are no nontrivial projectively equivalent pairs in $\mathcal{C}_{n}^{+}$
for $n\geq15$.
\end{thm}

Since for $n\geq4$,
\begin{eqnarray*}
 &  & (s_{1},\cdots,s_{n})\sim(t_{1},\cdots,t_{n})\\
 & \Leftrightarrow & (1,s_{2}/s_{1},\cdots,s_{n}/s_{1})\sim(1,t_{2}/t_{1},\cdots,t_{n}/t_{1})\\
 & \Leftrightarrow & (1,s_{2}/s_{1},s_{3}/s_{1},s_{i}/s_{1})\sim(1,t_{2}/t_{1},t_{3}/t_{1},t_{i}/t_{1})\text{ for any }4\leq i\leq n,
\end{eqnarray*}
$\mathcal{C}_{n}/\sim$ is determined by $\mathcal{C}_{4}/\sim$.
Now let us collect the solutions to the Diophantine equation
\[
(1,s_{1},s_{2},s_{3})\sim(1,t_{1},t_{2},t_{3})\text{ in }\mathcal{C}_{4}.
\]
\begin{defn}
\label{definition4}Let
\[
\mathcal{S}:=\{[s_{1},s_{2},s_{3},t_{1},t_{2},t_{3}]:(1,s_{1},s_{2},s_{3})\sim(1,t_{1},t_{2},t_{3})\text{ in }\mathcal{C}_{4}\}
\]
and $\mathcal{S}^{+}$ the isomorphic image of $\mathcal{S}$ under
$(\textup{Arg}/(2\pi))^{6}$.
\end{defn}

We notice that one solution in $\mathcal{S}$ can actually produce
many more solutions.
\begin{prop}
\label{proposition5}The groups $\textup{Gal}(\overline{\mathbb{Q}}/\mathbb{Q})$,
$\mathfrak{G}_{0}:=\left\langle (123)(456),(12)(45)\right\rangle \times\left\langle (14)(25)(36)\right\rangle \simeq\mathfrak{S}_{3}\times C_{2}$,
\begin{align*}
 & \mathfrak{G}_{s}:=\{g_{s}\text{ fixes }(s_{1},s_{2},s_{3})\text{ and maps }(t_{1},t_{2},t_{3})\text{ to some element of}\\
 & \{(t_{1},t_{2},t_{3})^{\pm1},(t_{1},t_{1}/t_{3},t_{1}/t_{2})^{\pm1},(t_{2}/t_{3},t_{2},t_{2}/t_{1})^{\pm1},(t_{3}/t_{2},t_{3}/t_{1},t_{3})^{\pm1}\}\}\simeq C_{2}^{3},
\end{align*}
and
\begin{align*}
 & \mathfrak{G}_{t}:=\{g_{t}\text{ fixes }(t_{1},t_{2},t_{3})\text{ and maps }(s_{1},s_{2},s_{3})\text{ to some element of}\\
 & \{(s_{1},s_{2},s_{3})^{\pm1},(s_{1},s_{1}/s_{3},s_{1}/s_{2})^{\pm1},(s_{2}/s_{3},s_{2},s_{2}/s_{1})^{\pm1},(s_{3}/s_{2},s_{3}/s_{1},s_{3})^{\pm1}\}\}\simeq C_{2}^{3}
\end{align*}
act on $\mathcal{S}$ naturally.
\end{prop}

\begin{proof}
If $[s_{1},s_{2},s_{3},t_{1},t_{2},t_{3}]\in\mathcal{S}$, then
\[
(1,s_{1},s_{2},s_{3})\sim(1,t_{1},t_{2},t_{3})\sim(1/t_{1},1,t_{2}/t_{1},t_{3}/t_{1})\sim(1,1/t_{1},t_{3}/t_{1},t_{2}/t_{1}),
\]
so $[s_{1},s_{2},s_{3},1/t_{1},t_{3}/t_{1},t_{2}/t_{1}]\in\mathcal{S}$.
The rest are straightforward.
\end{proof}
\begin{defn}
\label{definition6}Let $\mathfrak{G}:=\left\langle \mathfrak{G}_{0},\mathfrak{G}_{s},\mathfrak{G}_{t}\right\rangle $,
a group of order $768$. In $\mathcal{S}$, we write
\[
[s_{1},s_{2},s_{3},t_{1},t_{2},t_{3}]\approx[s_{1}',s_{2}',s_{3}',t_{1}',t_{2}',t_{3}']
\]
to denote that 
\[
[s_{1},s_{2},s_{3},t_{1},t_{2},t_{3}]=\sigma\cdot g\cdot[s_{1}',s_{2}',s_{3}',t_{1}',t_{2}',t_{3}']
\]
for some $\sigma\in\textup{Gal}(\overline{\mathbb{Q}}/\mathbb{Q})$
and $g\in\mathfrak{G}$. The corresponding equivalence relation in
$\mathcal{S}^{+}$ is also denoted by $\approx$.
\end{defn}

As we will see in a minute, the size of $\mathcal{S}$ is quite large,
so this equivalence relation is really helpful. With this simplification,
we are able to describe the structure of $\mathcal{S}^{+}/\approx$
now. We use the additive $\mathcal{S}^{+}$ rather than the multiplicative
$\mathcal{S}$ because $\theta$ is typographically simpler than $\exp(2\pi i\theta)$.
\begin{thm}
\label{theorem7}$\mathcal{S}^{+}/\approx$ can be classified as follows:
\begin{itemize}
\item $1$ three-parameter solution,
\item $3$ two-parameter solutions,
\item $41$ one-parameter solutions,
\item $730$ zero-parameter solutions.
\end{itemize}
The positive-parameter solutions are listed in Table \ref{table1},
where $x,y,z\in\mathbb{Q}/\mathbb{Z}$ can be taken arbitrarily as
long as $[s_{1},s_{2},s_{3},t_{1},t_{2},t_{3}]\in\mathcal{S}^{+}$.
The zero-parameter solutions are listed in Table \ref{table2}. (Since
it is too long, we put it at the end of this article to keep the text
uninterrupted.)
\end{thm}

\begin{longtable}{|>{\centering}p{2cm}|>{\centering}p{2cm}|>{\centering}p{2cm}|>{\centering}p{2cm}|>{\centering}p{2cm}|>{\centering}p{2cm}|}
\caption{The positive-parameter solutions in $\mathcal{S}^{+}/\approx$}
\label{table1}\tabularnewline
\hline 
$s_{1}$ & $s_{2}$ & $s_{3}$ & $t_{1}$ & $t_{2}$ & $t_{3}$\tabularnewline
\hline 
$x$ & $y$ & $z$ & $x$ & $y$ & $z$\tabularnewline
\hline 
$1/2$ & $x$ & $-x$ & $1/2$ & $y$ & $-y$\tabularnewline
\hline 
$x$ & $1/2+y$ & $1/2-y$ & $y$ & $1/2+x$ & $1/2-x$\tabularnewline
\hline 
$x$ & $y$ & $1/2-x+y$ & $2x$ & $2y$ & $1/2+y$\tabularnewline
\hline 
$x$ & $3x$ & $1/2+4x$ & $4x$ & $12x$ & $1/2+6x$\tabularnewline
\hline 
$1/6$ & $1/3$ & $x$ & $1/2+x$ & $1/3$ & $-2x$\tabularnewline
\hline 
$1/6$ & $x$ & $1/6+x$ & $1/3+x$ & $3x$ & $1/3+2x$\tabularnewline
\hline 
$1/6$ & $x$ & $1/6+x$ & $1/2+x$ & $2x$ & $1/2+3x$\tabularnewline
\hline 
$1/6$ & $x$ & $1/6+x$ & $1/2+x$ & $4x$ & $1/2+2x$\tabularnewline
\hline 
$1/6$ & $x$ & $1/2+x$ & $1/2$ & $1/6+x$ & $1/2+3x$\tabularnewline
\hline 
$1/6$ & $x$ & $2x$ & $1/6+x$ & $5/6+2x$ & $1/2+3x$\tabularnewline
\hline 
$1/6$ & $x$ & $2x$ & $1/6+x$ & $2/3+4x$ & $1/3+2x$\tabularnewline
\hline 
$1/6$ & $x$ & $2x$ & $1/3+2x$ & $3x$ & $6x$\tabularnewline
\hline 
$1/6$ & $x$ & $1/6+2x$ & $1/3+x$ & $2x$ & $1/3+4x$\tabularnewline
\hline 
$1/6$ & $x$ & $1/3+2x$ & $1/3+x$ & $2x$ & $1/2+3x$\tabularnewline
\hline 
$1/6$ & $x$ & $3x$ & $1/6+x$ & $5/6+3x$ & $1/3+2x$\tabularnewline
\hline 
$1/6$ & $x$ & $3x$ & $1/2+x$ & $1/6+2x$ & $1/3+4x$\tabularnewline
\hline 
$1/6$ & $x$ & $4x$ & $1/3+x$ & $5/6+4x$ & $1/2+2x$\tabularnewline
\hline 
$1/6$ & $x$ & $-x$ & $1/3+x$ & $1/6+x$ & $2/3+2x$\tabularnewline
\hline 
$1/6$ & $x$ & $1/2-x$ & $1/6+x$ & $x$ & $1/2+3x$\tabularnewline
\hline 
$1/6$ & $x$ & $2/3-x$ & $1/6+x$ & $5/6$ & $1/6-x$\tabularnewline
\hline 
$1/6$ & $x$ & $2/3-x$ & $1/6+2x$ & $2x$ & $1/3+4x$\tabularnewline
\hline 
$1/6$ & $x$ & $-2x$ & $1/3+x$ & $1/6+x$ & $1/2+3x$\tabularnewline
\hline 
$1/6$ & $x$ & $2/3-2x$ & $1/2+x$ & $5/6$ & $1/2-3x$\tabularnewline
\hline 
$1/6$ & $1/3+2x$ & $1/2+3x$ & $x$ & $3x$ & $1/3+4x$\tabularnewline
\hline 
$1/6$ & $1/3+2x$ & $1/2+4x$ & $x$ & $3x$ & $1/2+4x$\tabularnewline
\hline 
$1/6$ & $1/2+2x$ & $2/3+2x$ & $x$ & $4x$ & $1/2+3x$\tabularnewline
\hline 
$1/6$ & $1/2+2x$ & $1/2+6x$ & $x$ & $1/6+3x$ & $2/3+4x$\tabularnewline
\hline 
$1/3$ & $x$ & $5/6+2x$ & $1/3$ & $4x$ & $1/2+x$\tabularnewline
\hline 
$1/3$ & $x$ & $-x$ & $1/3$ & $1/6+x$ & $2/3-2x$\tabularnewline
\hline 
$1/3$ & $x$ & $5/6-x$ & $1/3$ & $5/6+2x$ & $1/3-4x$\tabularnewline
\hline 
$1/2$ & $x$ & $1/6-x$ & $1/2$ & $1/6+x$ & $1/2-3x$\tabularnewline
\hline 
$x$ & $2x$ & $3x$ & $1/6+x$ & $1/2+3x$ & $2/3+2x$\tabularnewline
\hline 
$x$ & $1/6+2x$ & $1/2+3x$ & $5/6+2x$ & $1/2+6x$ & $1/2+3x$\tabularnewline
\hline 
$x$ & $1/6+2x$ & $4x$ & $5/6+x$ & $1/3+4x$ & $1/2+2x$\tabularnewline
\hline 
$x$ & $1/6+2x$ & $2/3+5x$ & $5/6+x$ & $1/2+6x$ & $1/2+2x$\tabularnewline
\hline 
$x$ & $1/3+2x$ & $3x$ & $5/6+x$ & $1/6+2x$ & $1/3+4x$\tabularnewline
\hline 
$x$ & $3x$ & $9x$ & $1/6+x$ & $1/2+4x$ & $2/3+3x$\tabularnewline
\hline 
$x$ & $1/2+3x$ & $4x$ & $1/6+2x$ & $2/3+4x$ & $1/2+6x$\tabularnewline
\hline 
$x$ & $4x$ & $1/3+5x$ & $1/6+x$ & $1/2+6x$ & $2/3+4x$\tabularnewline
\hline 
$1/12$ & $1/4$ & $5/6$ & $x$ & $1/2$ & $-x$\tabularnewline
\hline 
$1/12$ & $x$ & $5/12+x$ & $1/6+x$ & $3x$ & $1/3+2x$\tabularnewline
\hline 
$1/12$ & $x$ & $5/12+x$ & $2/3+x$ & $1/3+2x$ & $2/3+4x$\tabularnewline
\hline 
$1/12$ & $x$ & $5/12+x$ & $1/3+2x$ & $6x$ & $1/2+3x$\tabularnewline
\hline 
$1/12$ & $1/3+2x$ & $3/4+2x$ & $x$ & $3x$ & $1/2+4x$\tabularnewline
\hline 
\end{longtable}

After introducing some necessary preliminaries in Section \ref{section2},
we will give the proofs of Theorem \ref{theorem7} and \ref{theorem3}
in Section \ref{section3} and \ref{section4}, respectively. As the
final results suggest, the proofs consist mostly of algorithms and
computations, which are realizable with the help of Mathematica 11.0
\citep{Mathematica}. The total running time is about seven hours
on the author's personal laptop. The source code will be attached
with this article.

Whenever we perform the computations on the computer, we will always
apply the most efficient algorithms we have. But when we describe
the algorithms in this article, we will simply outline their key ideas
in order to avoid too many cumbersome details. As a compensation,
we will provide some concrete examples for better illustration.

\section{\label{section2}Preliminaries}

\subsection{\label{section2.1}Vanishing sums of roots of unity}

Without loss of generality (since the group $\mu_{\infty}$ is divisible),
suppose that $[s_{1}^{2},s_{2}^{2},s_{3}^{2},t_{1}^{2},t_{2}^{2},t_{3}^{2}]\in\mathcal{S}$.
Define
\[
\begin{cases}
\begin{aligned} & S_{1}=\overline{s_{1}}s_{2}s_{3}, &  & S_{2}=s_{1}\overline{s_{2}}s_{3}, &  & S_{3}=s_{1}s_{2}\overline{s_{3}}, &  & T_{1}=\overline{t_{1}}t_{2}t_{3}, &  & T_{2}=t_{1}\overline{t_{2}}t_{3}, &  & T_{3}=t_{1}t_{2}\overline{t_{3}},\\
 & A_{1}=S_{3}T_{2}, &  & A_{2}=S_{1}T_{3}, &  & A_{3}=S_{2}T_{1}, &  & B_{1}=-\overline{S_{2}T_{3}}, &  & B_{2}=-\overline{S_{3}T_{1}}, &  & B_{3}=-\overline{S_{1}T_{2}},\\
 & A_{1}'=S_{3}\overline{T_{2}}, &  & A_{2}'=S_{1}\overline{T_{3}}, &  & A_{3}'=S_{2}\overline{T_{1}}, &  & B_{1}'=-\overline{S_{2}}T_{3}, &  & B_{2}'=-\overline{S_{3}}T_{1}, &  & B_{3}'=-\overline{S_{1}}T_{2},
\end{aligned}
\end{cases}
\]
where $\overline{\xi}=1/\xi$ represents the complex conjugate of
$\xi\in\mu_{\infty}$, then
\begin{eqnarray}
0 & = & [(s_{3}^{2}-s_{1}^{2})(s_{2}^{2}-1)(t_{2}^{2}-t_{1}^{2})(t_{3}^{2}-1)-(s_{2}^{2}-s_{1}^{2})(s_{3}^{2}-1)(t_{3}^{2}-t_{1}^{2})(t_{2}^{2}-1)]/(s_{1}s_{2}s_{3}t_{1}t_{2}t_{3})\nonumber \\
 & = & [(S_{3}+\overline{S_{3}})(T_{2}+\overline{T_{2}})+(S_{1}+\overline{S_{1}})(T_{3}+\overline{T_{3}})+(S_{2}+\overline{S_{2}})(T_{1}+\overline{T_{1}})]-\nonumber \\
 &  & [(S_{2}+\overline{S_{2}})(T_{3}+\overline{T_{3}})+(S_{3}+\overline{S_{3}})(T_{1}+\overline{T_{1}})+(S_{1}+\overline{S_{1}})(T_{2}+\overline{T_{2}})]\label{equation1}\\
 & = & A_{1}+A_{2}+A_{3}+B_{1}+B_{2}+B_{3}+A_{1}'+A_{2}'+A_{3}'+B_{1}'+B_{2}'+B_{3}'+\nonumber \\
 &  & \overline{A_{1}}+\overline{A_{2}}+\overline{A_{3}}+\overline{B_{1}}+\overline{B_{2}}+\overline{B_{3}}+\overline{A_{1}'}+\overline{A_{2}'}+\overline{A_{3}'}+\overline{B_{1}'}+\overline{B_{2}'}+\overline{B_{3}'}\nonumber 
\end{eqnarray}
and
\begin{equation}
\begin{cases}
A_{1}A_{2}A_{3}B_{1}B_{2}B_{3}A_{1}'A_{2}'A_{3}'B_{1}'B_{2}'B_{3}'=1,\\
A_{1}A_{2}A_{3}B_{1}B_{2}B_{3}=A_{1}'A_{2}'A_{3}'B_{1}'B_{2}'B_{3}'=-1,\\
A_{2}B_{3}A_{2}'B_{3}'=A_{3}B_{1}A_{3}'B_{1}'=A_{1}B_{2}A_{1}'B_{2}'=1,\\
A_{3}B_{2}\overline{A_{3}'B_{2}'}=A_{1}B_{3}\overline{A_{1}'B_{3}'}=A_{2}B_{1}\overline{A_{2}'B_{1}'}=1.
\end{cases}\label{equation2}
\end{equation}
Conversely, given $(A_{1},A_{2},A_{3},B_{1},B_{2},B_{3},A_{1}',A_{2}',A_{3}',B_{1}',B_{2}',B_{3}')$
satisfying (\ref{equation1}) and (\ref{equation2}), then
\begin{equation}
[s_{1}^{2},s_{2}^{2},s_{3}^{2},t_{1}^{2},t_{2}^{2},t_{3}^{2}]=[-A_{3}\overline{B_{2}'},-A_{1}\overline{B_{3}'},-A_{2}\overline{B_{1}'},-A_{2}B_{3}',-A_{3}B_{1}',-A_{1}B_{2}']\label{equation3}
\end{equation}
is almost a solution (which may be degenerate sometimes) in $\mathcal{S}$.

Thus we are led to study how $24$ roots of unity can sum to zero.
More generally, we consider the cyclotomic relation
\[
G=\sum_{i=1}^{m}\xi_{i}=0,
\]
where $\xi_{i}\in\mu_{\infty}$. This topic has been studied by many
people. Among them, the article of Poonen and Rubinstein \citep{MR1612877}
enlightens us quite a lot. In their work, in order to count the number
of interior intersection points and the number of regions made by
the diagonals inside a regular polygon, they classify the cyclotomic
relations up to $m=12$.

We will use their method substantially in the sequel. Let us first
collect some of their notations. We call $w(G)=m$ the weight of $G$.
$G$ is said to be minimal if $\sum_{i\in I}\xi_{i}\neq0$ for any
nonempty and proper subset $I\subseteq\{1,\cdots,m\}$. Any relation
can be decomposed as a sum of minimal relations (which may not be
unique). Write $\zeta_{n}=\exp(2\pi i/n)$ for any positive integer
$n$. For each prime $p$, let $R_{p}$ be the relation
\[
1+\zeta_{p}+\cdots+\zeta_{p}^{p-1}=0.
\]
Its minimality follows from the irreducibility of the cyclotomic polynomial.
Also we can ``rotate'' any relation by multiplying an arbitrary
root of unity to obtain a new relation. If we subtract $R_{3}$ from
$R_{5}$, cancel the $1$'s and incorporate the minus signs into the
roots of unity, we obtain a new minimal relation
\[
\zeta_{6}+\zeta_{6}^{5}+\zeta_{5}+\zeta_{5}^{2}+\zeta_{5}^{3}+\zeta_{5}^{4}=0,
\]
which we will denote $(R_{5}:R_{3})$. In general, if $G$ and $H_{1},\cdots,H_{j}$
are relations, we will use the notation $(G:H_{1},\cdots,H_{j})$
to denote any relation obtained by rotating the $H_{i}$ so that each
shares exactly one root of unity with $G$ which is different for
each $i$, subtracting them from $G$, and incorporating the minus
signs into the roots of unity. For notational convenience, we will
write $(R_{5}:4R_{3})$ for $(R_{5}:R_{3},R_{3},R_{3},R_{3})$, for
example. Note that in general there will be many relations of type
$(G:H_{1},\cdots,H_{j})$ up to rotational equivalence.

The following results will be needed in our arguments.
\begin{lem}
\label{lemma8}\textup{{[}\citealp[Theorem 1]{MR0191892}; \citealp[Lemma 1]{MR1612877}{]}}
If the relation $G=\sum_{i=1}^{m}\xi_{i}=0$ is minimal, then there
are distinct primes $p_{1}<\cdots<p_{k}\leq m$ so that each $\xi_{i}$
is a $p_{1}\cdots p_{k}$-th root of unity, after the relation has
been suitably rotated.
\end{lem}

\begin{lem}
\textup{\label{lemma9}\citep[Lemma 2]{MR1612877}} The only minimal
relations (up to rotation) involving only the $2p$-th roots of unity,
for $p$ prime, are $R_{2}$ and $R_{p}$.
\end{lem}

\begin{lem}
\textup{\label{lemma10}\citep[Lemma 3]{MR1612877}} Suppose $G$
is a minimal relation, and $p_{1}<\cdots<p_{k}$ are picked as in
Lemma \ref{lemma8} with $p_{1}=2$ and $p_{k}$ minimal. If $w(G)<2p_{k}$,
then $G$ (or a rotation) is of the form $(R_{p_{k}}:H_{1},\cdots,H_{j})$
where the $H_{i}$ are minimal relations not equal to $R_{2}$ and
involving only $p_{1}\cdots p_{k-1}$-th roots of unity, such that
$j<p_{k}$ and $\sum_{i=1}^{j}[w(H_{i})-2]=w(G)-p_{k}$.
\end{lem}

\begin{lem}
\textup{\label{lemma11}(Modified from \citep[Lemma 4]{MR1612877})}
Suppose the relation $G$ is stable under complex conjugation. If
$R_{p}$, for $p$ prime, occurs in $G$, then it is itself stable
under complex conjugation, or can be paired with another $R_{p}$
which is its complex conjugate.
\end{lem}

\begin{lem}
\textup{\label{lemma12}\citep[Lemma 5]{MR1612877}} Let $G$ be a
minimal relation of type $(R_{p}:H_{1},\cdots,H_{j})$, $p\geq5$,
where the $H_{i}$ involve roots of unity of order prime to $p$,
and $j<p$. If $G$ is stable under complex conjugation, then the
particular rotation of $R_{p}$ from which the $H_{i}$ were ``subtracted''
is also stable (and hence so is the collection of the relations subtracted).
\end{lem}

\addtocounter{table}{-1}
\begin{lem}
\textup{\label{lemma13}\citep[Theorem 3]{MR1612877}} The minimal
relations of weight up to $12$ are as follows:

\begin{longtable}{|c|c|c|c|c|c|}
\hline 
$w=2,3,5,6$ & $R_{2}$ & $R_{3}$ & $R_{5}$ & $(R_{5}:R_{3})$ & \tabularnewline
\hline 
$w=7,8$ & $(R_{5}:2R_{3})$ & $R_{7}$ & $(R_{5}:3R_{3})$ & $(R_{7}:R_{3})$ & \tabularnewline
\hline 
$w=9,10$ & $(R_{5}:4R_{3})$ & $(R_{7}:2R_{3})$ & $(R_{7}:3R_{3})$ & $(R_{7}:R_{5})$ & \tabularnewline
\hline 
$w=11$ & $(R_{7}:4R_{3})$ & $(R_{7}:R_{5},R_{3})$ & $(R_{7}:(R_{5}:R_{3}))$ & $R_{11}$ & \tabularnewline
\hline 
$w=12$ & $(R_{7}:5R_{3})$ & $(R_{7}:R_{5},2R_{3})$ & $(R_{7}:(R_{5}:R_{3}),R_{3})$ & $(R_{7}:(R_{5}:2R_{3}))$ & $(R_{11}:R_{3})$\tabularnewline
\hline 
\end{longtable}
\end{lem}

The reason why we take $[s_{1}^{2},s_{2}^{2},s_{3}^{2},t_{1}^{2},t_{2}^{2},t_{3}^{2}]\in\mathcal{S}$
rather than the more natural $[s_{1},s_{2},s_{3},t_{1},t_{2},t_{3}]\in\mathcal{S}$
at the beginning of this section is to make the relation (\ref{equation1})
stable under complex conjugation, so that Lemma \ref{lemma11} and
\ref{lemma12} can be applied.

\subsection{\label{section2.2}Smith normal form}

\addtocounter{equation}{-1}

The relations (\ref{equation2}) can be reduced to
\[
\begin{cases}
A_{1}A_{2}A_{3}B_{1}B_{2}B_{3}=A_{1}'A_{2}'A_{3}'B_{1}'B_{2}'B_{3}'=-1,\\
A_{2}B_{3}A_{2}'B_{3}'=A_{3}B_{1}A_{3}'B_{1}'=1,\\
A_{3}B_{2}\overline{A_{3}'B_{2}'}=A_{1}B_{3}\overline{A_{1}'B_{3}'}=1,
\end{cases}
\]
which are equivalent to the linear equations
\begin{equation}
\begin{cases}
a_{1}+a_{2}+a_{3}+b_{1}+b_{2}+b_{3}=a_{1}'+a_{2}'+a_{3}'+b_{1}'+b_{2}'+b_{3}'=1/2,\\
a_{2}+b_{3}+a_{2}'+b_{3}'=a_{3}+b_{1}+a_{3}'+b_{1}'=0,\\
a_{3}+b_{2}-a_{3}'-b_{2}'=a_{1}+b_{3}-a_{1}'-b_{3}'=0,
\end{cases}\tag{2\ensuremath{'}}\label{equation2'}
\end{equation}
where the lower case unknowns are the images of corresponding upper
case unknowns under $\textup{Arg}/(2\pi)$. Note that the numbers
on the right-hand side of (\ref{equation2'}) belong to $\mathbb{Q}/\mathbb{Z}$.
An important tool to deal with such equations is the following well-known
result in elementary algebra.

\addtocounter{equation}{1}
\begin{thm}
\label{theorem14}Let $M$ be an $m\times n$ matrix of rank $r$
over $\mathbb{Z}$. There exist invertible $m\times m$ and $n\times n$
matrices $S$ and $T$ such that the product $SMT$ is of the form
$D=\begin{pmatrix}D_{0} & 0\\
0 & 0
\end{pmatrix}$, where $D_{0}$ is an $r\times r$ diagonal matrix with positive
entries $d_{i}$ such that $d_{i}\mid d_{i+1}$ for any $1\leq i<r$.
$D$ is called the Smith normal form of $M$.
\end{thm}

\subsection{\label{section2.3}Cyclotomic units}

We should note that although the method developed by Poonen and Rubinstein
works perfectly for their own problem, it is not enough for us. The
main reason is, briefly speaking, our weight $24$ is much larger
than their weight $12$. More details will be explained in Section
\ref{section3.7}.

If $[s_{1},s_{2},s_{3},t_{1},t_{2},t_{3}]\in\mathcal{S}$, then
\begin{equation}
\frac{(s_{3}/s_{1}-1)(s_{2}-1)}{(s_{2}/s_{1}-1)(s_{3}-1)}=\frac{(t_{3}/t_{1}-1)(t_{2}-1)}{(t_{2}/t_{1}-1)(t_{3}-1)}\in\mathbb{R}.\label{equation4}
\end{equation}
The numbers of the form $\xi-1$ with $\xi\in\mu_{\infty}$ are related
to the so-called cyclotomic units (whose definition is irrelevant
to us). The following fundamental result regarding the cyclotomic
units will be needed in Algorithm \ref{algorithm37}.
\begin{thm}
\textup{\label{theorem15}{[}\citealp{MR0201414}; \citealp{MR0299585};
\citealp[Theorem 8.9]{MR1421575}{]}} Let $n$ be a positive integer
and $E_{n}$ the multiplicative abelian group generated by
\[
\{e_{k}=\left|\zeta_{n}^{k}-1\right|:1\leq k<n\}.
\]
Then for any relation $R$ in $E_{n}$, $R^{2}$ is a $\mathbb{Z}$-linear
combination of the relations
\begin{equation}
e_{k}=e_{-k}\label{equation5}
\end{equation}
and
\begin{equation}
e_{k}=\prod_{pi=k}e_{i}\text{ for prime }p\mid n.\label{equation6}
\end{equation}
\end{thm}

In fact, what we will use is a weaker form of Theorem \ref{theorem15}:
any relation in $E_{n}\otimes_{\mathbb{Z}}\mathbb{Q}$ is a $\mathbb{Q}$-linear
combination of the relations (\ref{equation5}) and (\ref{equation6}).

\section{\label{section3}Proof of Theorem \texorpdfstring{\ref{theorem7}}{7}}

We decompose the relation (\ref{equation1}) into a sum of minimal
relations, and then consider the largest prime $p$ involved. Since
(\ref{equation1}) has weight $24$, by Lemma \ref{lemma8}, we have
\[
p\in\{2,3,5,7,11,13,17,19,23\}.
\]
We divide the proof into four parts: $p=17,19,23$, $p=11,13$, $p=2,3$,
and finally the most tough $p=5,7$. In Section \ref{section3.7},
we will explain why different strategies are required for different
$p$.

\subsection{\label{section3.1}Symmetry of the relations \texorpdfstring{(\ref{equation1})
and (\ref{equation2})}{(1) and (2)}}

Since the associated matrix of (\ref{equation2'})\[\fixTABwidth{T}\parenMatrixstack{
1 & 1 & 1 & 1 & 1 & 1 & 0 & 0 & 0 & 0 & 0 & 0\\
0 & 0 & 0 & 0 & 0 & 0 & 1 & 1 & 1 & 1 & 1 & 1\\
0 & 1 & 0 & 0 & 0 & 1 & 0 & 1 & 0 & 0 & 0 & 1\\
0 & 0 & 1 & 1 & 0 & 0 & 0 & 0 & 1 & 1 & 0 & 0\\
0 & 0 & 1 & 0 & 1 & 0 & 0 & 0 & -1 & 0 & -1 & 0\\
1 & 0 & 0 & 0 & 0 & 1 & -1 & 0 & 0 & 0 & 0 & -1
}\]has rank $6$, we can fix six unknowns and then try to solve the remaining
six. However, there are $\binom{12}{6}=924$ ways to do that, so we
need to reduce the number of cases first. We do that by considering
the following natural group actions of $\mathfrak{G}$.
\begin{itemize}
\item $\mathfrak{G}$ permutes the $24$-element set
\[
\{A_{1}^{\pm1},A_{2}^{\pm1},A_{3}^{\pm1},-B_{1}^{\pm1},-B_{2}^{\pm1},-B_{3}^{\pm1},A_{1}'^{\pm1},A_{2}'^{\pm1},A_{3}'^{\pm1},-B_{1}'^{\pm1},-B_{2}'^{\pm1},-B_{3}'^{\pm1}\}.
\]
This action is transitive and faithful.
\item Modulo negation, $\mathfrak{G}$ permutes the $24$-element set
\[
U_{24}:=\{A_{1}^{\pm1},A_{2}^{\pm1},A_{3}^{\pm1},B_{1}^{\pm1},B_{2}^{\pm1},B_{3}^{\pm1},A_{1}'^{\pm1},A_{2}'^{\pm1},A_{3}'^{\pm1},B_{1}'^{\pm1},B_{2}'^{\pm1},B_{3}'^{\pm1}\}.
\]
This action is equivalent to the previous one.
\item Modulo negation and inversion, $\mathfrak{G}$ permutes the $12$-element
set
\[
U_{12}:=\{A_{1},A_{2},A_{3},B_{1},B_{2},B_{3},A_{1}',A_{2}',A_{3}',B_{1}',B_{2}',B_{3}'\}.
\]
This action is transitive but not faithful. The kernel is
\[
\{[s_{1},s_{2},s_{3},t_{1},t_{2},t_{3}]\mapsto[s_{1},s_{2},s_{3},t_{1},t_{2},t_{3}]^{\pm1}\}\simeq C_{2}.
\]
\item $\mathfrak{G}$ permutes the $\binom{12}{n}$-element set
\[
U_{12,n}:=\{\text{the }n\text{-element subsets of }U_{12}\}.
\]
\end{itemize}
The following two lemmas will be essentially used in Algorithm \ref{algorithm27}.

\addtocounter{table}{-1}
\begin{lem}
\label{lemma16}The action of $\mathfrak{G}$ on $U_{12,6}$ has the
following orbits:

\begin{longtable}{|c|c|c|c|c|}
\hline 
$O_{1}$ & $\{A_{1},A_{2},A_{3},B_{1},B_{2},A_{1}'\}$ & $192$ & $(1,1,1,1,1,1)$ & \tabularnewline
\hline 
$O_{2}$ & $\{A_{1},A_{2},A_{3},B_{1},B_{2},A_{3}'\}$ & $96$ & $(1,1,1,1,1,1)$ & \tabularnewline
\hline 
$O_{3}$ & $\{A_{1},A_{2},A_{3},B_{1},B_{2},B_{1}'\}$ & $192$ & $(1,1,1,1,1,1)$ & \tabularnewline
\hline 
$O_{4}$ & $\{A_{1},A_{2},A_{3},B_{1},B_{2},B_{3}'\}$ & $32$ & $(1,1,1,1,1,1)$ & \tabularnewline
\hline 
$O_{5}$ & $\{A_{1},A_{2},A_{3},B_{1},A_{1}',A_{2}'\}$ & $48$ & $(1,1,1,1,1,2)$ & \tabularnewline
\hline 
$O_{6}$ & $\{A_{1},A_{2},A_{3},B_{1},A_{1}',B_{1}'\}$ & $24$ & $(1,1,1,1,1,2)$ & \tabularnewline
\hline 
$O_{7}$ & $\{A_{1},A_{2},B_{1},B_{2},A_{1}',A_{2}'\}$ & $24$ & $(1,1,1,1,1,2)$ & \tabularnewline
\hline 
$O_{8}$ & $\{A_{1},A_{2},B_{1},B_{2},A_{1}',B_{1}'\}$ & $24$ & $(1,1,1,1,1,2)$ & \tabularnewline
\hline 
$O_{9}$ & $\{A_{1},A_{2},A_{3},A_{1}',A_{2}',A_{3}'\}$ & $2$ & $(1,1,1,1,2,2)$ & \tabularnewline
\hline 
$O_{10}$ & $\{A_{1},A_{2},A_{3},B_{1},B_{2},B_{3}\}$ & $32$ & $(1,1,1,1,1,0)$ & $A_{1}A_{2}A_{3}B_{1}B_{2}B_{3}=-1$\tabularnewline
\hline 
$O_{11}$ & $\{A_{1},A_{2},A_{3},B_{1},A_{2}',A_{3}'\}$ & $24$ & $(1,1,1,1,1,0)$ & $(A_{3}B_{1}A_{3}'B_{1}')(A_{2}B_{1}\overline{A_{2}'B_{1}'})=1$\tabularnewline
\hline 
$O_{12}$ & $\{A_{1},A_{2},A_{3},B_{1},A_{2}',B_{1}'\}$ & $48$ & $(1,1,1,1,1,0)$ & $A_{2}B_{1}\overline{A_{2}'B_{1}'}=1$\tabularnewline
\hline 
$O_{13}$ & $\{A_{1},A_{2},B_{1},B_{2},A_{1}',B_{2}'\}$ & $24$ & $(1,1,1,1,1,0)$ & $A_{1}B_{2}A_{1}'B_{2}'=1$\tabularnewline
\hline 
$O_{14}$ & $\{A_{1},A_{2},B_{1},B_{3},A_{1}',A_{2}'\}$ & $48$ & $(1,1,1,1,1,0)$ & $(A_{2}B_{3}A_{2}'B_{3}')(A_{1}B_{3}\overline{A_{1}'B_{3}'})=1$\tabularnewline
\hline 
$O_{15}$ & $\{A_{1},A_{2},B_{1},B_{3},A_{1}',B_{1}'\}$ & $24$ & $(1,1,1,1,1,0)$ & $(A_{2}B_{3}A_{2}'B_{3}')(A_{1}B_{3}\overline{A_{1}'B_{3}'})(A_{2}B_{1}\overline{A_{2}'B_{1}'})=1$\tabularnewline
\hline 
$O_{16}$ & $\{A_{1},A_{2},B_{1},B_{3},A_{1}',B_{3}'\}$ & $48$ & $(1,1,1,1,1,0)$ & $A_{1}B_{3}\overline{A_{1}'B_{3}'}=1$\tabularnewline
\hline 
$O_{17}$ & $\{A_{1},A_{2},B_{1},B_{3},A_{2}',B_{3}'\}$ & $24$ & $(1,1,1,1,1,0)$ & $A_{2}B_{3}A_{2}'B_{3}'=1$\tabularnewline
\hline 
$O_{18}$ & $\{A_{1},A_{2},B_{1},A_{1}',A_{2}',B_{1}'\}$ & $12$ & $(1,1,1,1,2,0)$ & $A_{2}B_{1}\overline{A_{2}'B_{1}'}=1$\tabularnewline
\hline 
$O_{19}$ & $\{A_{1},A_{2},B_{3},A_{1}',A_{2}',B_{3}'\}$ & $6$ & $(1,1,1,1,0,0)$ & $A_{2}B_{3}A_{2}'B_{3}'=A_{1}B_{3}\overline{A_{1}'B_{3}'}=1$\tabularnewline
\hline 
\end{longtable}
\end{lem}

Let us take $O_{18}$ as an example to explain each column of the
table. The second and third columns say that the orbit containing
$\{a_{1},a_{2},b_{1},a_{1}',a_{2}',b_{1}'\}$ has $12$ elements.
If we consider $\{a_{1},a_{2},b_{1},a_{1}',a_{2}',b_{1}'\}$ as constants
and $\{a_{3},b_{2},b_{3},a_{3}',b_{2}',b_{3}'\}$ as unknowns, then
(\ref{equation2'}) becomes
\[
\begin{pmatrix}1 & 1 & 1 & 0 & 0 & 0\\
0 & 0 & 0 & 1 & 1 & 1\\
0 & 0 & 1 & 0 & 0 & 1\\
1 & 0 & 0 & 1 & 0 & 0\\
1 & 1 & 0 & -1 & -1 & 0\\
0 & 0 & 1 & 0 & 0 & -1
\end{pmatrix}\begin{pmatrix}a_{3}\\
b_{2}\\
b_{3}\\
a_{3}'\\
b_{2}'\\
b_{3}'
\end{pmatrix}=\begin{pmatrix}1/2-a_{1}-a_{2}-b_{1}\\
1/2-a_{1}'-a_{2}'-b_{1}'\\
-a_{2}-a_{2}'\\
-b_{1}-b_{1}'\\
0\\
-a_{1}+a_{1}'
\end{pmatrix}.
\]
By Theorem \ref{theorem14},
\[
\begin{pmatrix}-1 & 0 & 1 & 1 & 1 & 0\\
1 & 1 & -1 & -1 & 0 & 0\\
0 & 0 & 1 & 0 & 0 & 0\\
0 & 1 & 0 & 0 & 0 & 0\\
-1 & 1 & 1 & 0 & 1 & 0\\
-1 & 1 & 0 & 0 & 1 & 1
\end{pmatrix}^{-1}\begin{pmatrix}1 & 0 & 0 & 0 & 0 & 0\\
0 & 1 & 0 & 0 & 0 & 0\\
0 & 0 & 1 & 0 & 0 & 0\\
0 & 0 & 0 & 1 & 0 & 0\\
0 & 0 & 0 & 0 & 2 & 0\\
0 & 0 & 0 & 0 & 0 & 0
\end{pmatrix}\begin{pmatrix}1 & 0 & 0 & 0 & -1 & 1\\
0 & 1 & 0 & 0 & 0 & -1\\
0 & 0 & 1 & 0 & -1 & 0\\
0 & 0 & 0 & 1 & -1 & -1\\
0 & 0 & 0 & 0 & 0 & 1\\
0 & 0 & 0 & 0 & 1 & 0
\end{pmatrix}^{-1}\begin{pmatrix}a_{3}\\
b_{2}\\
b_{3}\\
a_{3}'\\
b_{2}'\\
b_{3}'
\end{pmatrix}=\begin{pmatrix}1/2-a_{1}-a_{2}-b_{1}\\
1/2-a_{1}'-a_{2}'-b_{1}'\\
-a_{2}-a_{2}'\\
-b_{1}-b_{1}'\\
0\\
-a_{1}+a_{1}'
\end{pmatrix}.
\]
The fourth column says that the diagonal of the Smith normal form
is $(1,1,1,1,2,0)$. Multiplying the inverse of the first matrix on
both sides, we get
\[
\begin{pmatrix}1 & 0 & 0 & 0 & 0 & 0\\
0 & 1 & 0 & 0 & 0 & 0\\
0 & 0 & 1 & 0 & 0 & 0\\
0 & 0 & 0 & 1 & 0 & 0\\
0 & 0 & 0 & 0 & 2 & 0\\
0 & 0 & 0 & 0 & 0 & 0
\end{pmatrix}\begin{pmatrix}1 & 0 & 0 & 0 & -1 & 1\\
0 & 1 & 0 & 0 & 0 & -1\\
0 & 0 & 1 & 0 & -1 & 0\\
0 & 0 & 0 & 1 & -1 & -1\\
0 & 0 & 0 & 0 & 0 & 1\\
0 & 0 & 0 & 0 & 1 & 0
\end{pmatrix}^{-1}\begin{pmatrix}a_{3}\\
b_{2}\\
b_{3}\\
a_{3}'\\
b_{2}'\\
b_{3}'
\end{pmatrix}=\begin{pmatrix}1/2+a_{1}-a_{2}'-b_{1}'\\
-a_{1}-a_{1}'\\
-a_{2}-a_{2}'\\
1/2-a_{1}'-a_{2}'-b_{1}'\\
a_{1}+b_{1}-a_{1}'-2a_{2}'-b_{1}'\\
a_{2}+b_{1}-a_{2}'-b_{1}'
\end{pmatrix}.
\]
The fifth column says that (\ref{equation2'}) has a solution if and
only if $a_{2}+b_{1}-a_{2}'-b_{1}'=0$. Assuming this condition is
satisfied, then
\[
\begin{pmatrix}1 & 0 & 0 & 0 & -1 & 1\\
0 & 1 & 0 & 0 & 0 & -1\\
0 & 0 & 1 & 0 & -1 & 0\\
0 & 0 & 0 & 1 & -1 & -1\\
0 & 0 & 0 & 0 & 0 & 1\\
0 & 0 & 0 & 0 & 1 & 0
\end{pmatrix}^{-1}\begin{pmatrix}a_{3}\\
b_{2}\\
b_{3}\\
a_{3}'\\
b_{2}'\\
b_{3}'
\end{pmatrix}=\begin{pmatrix}1/2+a_{1}-a_{2}'-b_{1}'\\
-a_{1}-a_{1}'\\
-a_{2}-a_{2}'\\
1/2-a_{1}'-a_{2}'-b_{1}'\\
(a_{1}+b_{1}-a_{1}'-b_{1}')/2-a_{2}'\\
x
\end{pmatrix}\text{ or }\begin{pmatrix}1/2+a_{1}-a_{2}'-b_{1}'\\
-a_{1}-a_{1}'\\
-a_{2}-a_{2}'\\
1/2-a_{1}'-a_{2}'-b_{1}'\\
1/2+(a_{1}+b_{1}-a_{1}'-b_{1}')/2-a_{2}'\\
x
\end{pmatrix},
\]
where $x$ is a free variable. Multiplying the inverse of the first
matrix on both sides, we get
\[
\begin{pmatrix}a_{3}\\
b_{2}\\
b_{3}\\
a_{3}'\\
b_{2}'\\
b_{3}'
\end{pmatrix}=\begin{pmatrix}1/2+(a_{1}-b_{1}+a_{1}'-b_{1}')/2+x\\
-a_{1}-a_{1}'-x\\
-(a_{1}+b_{1}-a_{1}'-b_{1}')/2-a_{2}\\
1/2-(a_{1}+b_{1}+a_{1}'+b_{1}')/2-x\\
x\\
(a_{1}+b_{1}-a_{1}'-b_{1}')/2-a_{2}'
\end{pmatrix}\text{ or }\begin{pmatrix}(a_{1}-b_{1}+a_{1}'-b_{1}')/2+x\\
-a_{1}-a_{1}'-x\\
1/2-(a_{1}+b_{1}-a_{1}'-b_{1}')/2-a_{2}\\
-(a_{1}+b_{1}+a_{1}'+b_{1}')/2-x\\
x\\
1/2+(a_{1}+b_{1}-a_{1}'-b_{1}')/2-a_{2}'
\end{pmatrix}.
\]
\begin{lem}
\label{lemma17}$ $
\begin{itemize}
\item Every element of $U_{12,7}$ has a $6$-element subset that belongs
to $O_{1}$, $O_{3}$, $O_{5}$, $O_{7}$, $O_{11}$, or $O_{14}$.
\item Every element of $U_{12,8}$ has a $6$-element subset that belongs
to $O_{1}$, $O_{2}$, $O_{6}$, $O_{7}$, or $O_{15}$.
\item Every element of $U_{12,9}$ has a $6$-element subset that belongs
to $O_{1}$.
\end{itemize}
\end{lem}

\subsection{\label{section3.2}\texorpdfstring{$\boldsymbol{p=17,19,23}$}{$p=17,19,23$}}

Let $\mu_{n}$ be the collection of $n$-th roots of unity. If $n=\prod_{i=1}^{k}p_{i}^{n_{i}}$
is the prime factorization of $n$, then every $\xi\in\mu_{n}$ can
be uniquely written as a product $\xi=\prod_{i=1}^{k}\xi_{i}$, where
$\xi_{i}\in\mu_{p_{i}^{n_{i}}}$ for any $1\leq i\leq k$. We call
$\xi_{i}$ the $p_{i}$-component of $\xi$.

\addtocounter{table}{-1}
\begin{lem}
\label{lemma18}The relations $G=0$ with $w(G)=24$ and $p=17,19,23$
are as follows:

\begin{longtable}{|c|c|c|}
\hline 
$(R_{23}:R_{3})$ &  & \tabularnewline
\hline 
$R_{19}+R_{5}$ & $R_{19}+R_{3}+R_{2}$ & $(R_{19}:R_{3})+2R_{2}$\tabularnewline
\hline 
$(R_{19}:2R_{3})+R_{3}$ & $(R_{19}:R_{5})+R_{2}$ & $(R_{19}:3R_{3})+R_{2}$\tabularnewline
\hline 
$(R_{19}:R_{7})$ & $(R_{19}:(R_{5}:2R_{3}))$ & $\boldsymbol{(R_{19}:(R_{5}:R_{3}),R_{3})^{*}}$\tabularnewline
\hline 
$(R_{19}:R_{5},2R_{3})$ & $(R_{19}:5R_{3})$ & \tabularnewline
\hline 
$R_{17}+R_{7}$ & $R_{17}+(R_{5}:2R_{3})$ & $R_{17}+R_{5}+R_{2}$\tabularnewline
\hline 
$R_{17}+R_{3}+2R_{2}$ & $(R_{17}:R_{3})+(R_{5}:R_{3})$ & $(R_{17}:R_{3})+2R_{3}$\tabularnewline
\hline 
$(R_{17}:R_{3})+3R_{2}$ & $(R_{17}:2R_{3})+R_{5}$ & $(R_{17}:2R_{3})+R_{3}+R_{2}$\tabularnewline
\hline 
$(R_{17}:R_{5})+2R_{2}$ & $(R_{17}:3R_{3})+2R_{2}$ & $\boldsymbol{(R_{17}:(R_{5}:R_{3}))+R_{3}^{*}}$\tabularnewline
\hline 
$\boldsymbol{(R_{17}:R_{5},R_{3})+R_{3}^{*}}$ & $(R_{17}:4R_{3})+R_{3}$ & $(R_{17}:R_{7})+R_{2}$\tabularnewline
\hline 
$(R_{17}:(R_{5}:2R_{3}))+R_{2}$ & $\boldsymbol{(R_{17}:(R_{5}:R_{3}),R_{3})+R_{2}^{*}}$ & $(R_{17}:R_{5},2R_{3})+R_{2}$\tabularnewline
\hline 
$(R_{17}:5R_{3})+R_{2}$ & $(R_{17}:(R_{7}:2R_{3}))$ & $(R_{17}:(R_{5}:4R_{3}))$\tabularnewline
\hline 
$\boldsymbol{(R_{17}:(R_{7}:R_{3}),R_{3})^{*}}$ & $\boldsymbol{(R_{17}:(R_{5}:3R_{3}),R_{3})^{*}}$ & $(R_{17}:R_{7},2R_{3})$\tabularnewline
\hline 
$(R_{17}:(R_{5}:2R_{3}),2R_{3})$ & $\boldsymbol{(R_{17}:(R_{5}:R_{3}),R_{5})^{*}}$ & $\boldsymbol{(R_{17}:(R_{5}:R_{3}),3R_{3})^{*}}$\tabularnewline
\hline 
$(R_{17}:2R_{5},R_{3})$ & $(R_{17}:R_{5},4R_{3})$ & $(R_{17}:7R_{3})$\tabularnewline
\hline 
\end{longtable}
\end{lem}

\begin{proof}
Since $w(G)<2p$, by Lemma \ref{lemma10}, $G$ must be of the form
\begin{equation}
G=(R_{p}:H_{1},\cdots,H_{j})+G_{1}+\cdots+G_{k},\label{equation7}
\end{equation}
where $\sum_{i=1}^{j}[w(H_{i})-2]+\sum_{i=1}^{k}w(G_{i})=w(G)-p\leq7$.
Since the minimal relations of weight up to $9$ are given in Lemma
\ref{lemma13}, we are done.
\end{proof}
Now we want to find all the solutions of (\ref{equation1}) and (\ref{equation2})
by assuming that (\ref{equation1}) is of the form $G$, where $G$
is one of the entries listed in Lemma \ref{lemma18}.

By Lemma \ref{lemma12}, the marked entries cannot be stable under
complex conjugation, so we can ignore them. By Lemma \ref{lemma11},
if $G\neq(R_{17}:R_{3})+(R_{5}:R_{3})$, then
\begin{condition}
\label{condition19}$G$ can be decomposed as a sum of minimal relations
such that each minimal relation is itself stable under complex conjugation,
or can be paired with another minimal relation which is its complex
conjugate.
\end{condition}

The following two lemmas show that this is also true for $G=(R_{17}:R_{3})+(R_{5}:R_{3})$.
\begin{lem}
\label{lemma20}Let $G=\xi_{1}G_{1}+\xi_{2}G_{2}$ be a relation which
is stable under complex conjugation, where $\xi_{i}\in\mu_{\infty}$
and $G_{i}$ are minimal relations. Assume that all the terms of $G_{i}$
belong to $\mu_{n_{i}}$. If $\xi_{1}G_{1}$ and $\xi_{2}G_{2}$ are
not conjugate to each other, then $\xi_{i}\in\mu_{2n_{i}}$.
\end{lem}

\begin{proof}
By the assumptions, two terms of $G_{i}$ must be conjugate to each
other, so $\xi_{i}^{2}\in\mu_{n_{i}}$.
\end{proof}
\begin{lem}
\label{lemma21}If $G=(R_{17}:R_{3})+(R_{5}:R_{3})$ is stable under
complex conjugation, then both $(R_{17}:R_{3})$ and $(R_{5}:R_{3})$
are stable under complex conjugation.
\end{lem}

\begin{proof}
Assume that
\[
G=\xi_{1}(\zeta_{6}+\zeta_{6}^{5}+\zeta_{17}+\cdots+\zeta_{17}^{16})+\xi_{2}(\zeta_{6}+\zeta_{6}^{5}+\zeta_{5}+\cdots+\zeta_{5}^{4})
\]
is stable under complex conjugation. By Lemma \ref{lemma20}, $\xi_{1}\in\mu_{204}$
and $\xi_{2}\in\mu_{60}$. Since the terms of $G$ with nontrivial
$17$-components are stable under complex conjugation, $\xi_{1}$
must have a trivial $17$-component and also $\xi_{1}^{2}\in\mu_{17}$.
Therefore, $\xi_{1}=\pm1$, and consequently, $\xi_{2}=\pm1$.
\end{proof}
The longest minimal relation occurring in $G$ has weight at least
$17$, so at least $9$ elements of $U_{12}$ are known. It is easy
to see (directly or by Lemma \ref{lemma17}) that if at least $9$
elements of $U_{12}$ are given, then $U_{12}$ can be completely
determined. In particular, if at least $9$ elements of $U_{12}$
belong to $\mu_{n}$, then all elements of $U_{12}$ belong to $\mu_{n}$.
Therefore, we are led to consider the $q$-components of $G$ for
prime $q$.
\begin{defn}
\label{definition22}Let $G=\sum_{i=1}^{m}\xi_{i}=0$ be a relation
such that all the $q$-components of $\xi_{i}$ belong to $\mu_{q}$
(i.e., higher powers of $q$ are not involved). Let $n_{k}$ be the
number of $\zeta_{q}^{k}$ in the $q$-components of $\xi_{i}$. We
call $(n_{k})_{k=0}^{q-1}$ the $q$-type of $G$. We say that $(n_{j+k})_{k=0}^{q-1}$
is a rotation of $(n_{k})_{k=0}^{q-1}$ for any $j\in\mathbb{Z}/q\mathbb{Z}$,
and $(n_{jk})_{k=0}^{q-1}$ is a Galois conjugate of $(n_{k})_{k=0}^{q-1}$
for any $j\in(\mathbb{Z}/q\mathbb{Z})^{\times}$.
\end{defn}

\addtocounter{equation}{-1}

Consider the $q$-component of (\ref{equation2'}):
\begin{equation}
\begin{cases}
a_{1}+a_{2}+a_{3}+b_{1}+b_{2}+b_{3}=a_{1}'+a_{2}'+a_{3}'+b_{1}'+b_{2}'+b_{3}'=\begin{cases}
1/2 & \text{if }q=2,\\
0 & \text{if }q\neq2,
\end{cases}\\
a_{2}+b_{3}+a_{2}'+b_{3}'=a_{3}+b_{1}+a_{3}'+b_{1}'=0,\\
a_{3}+b_{2}-a_{3}'-b_{2}'=a_{1}+b_{3}-a_{1}'-b_{3}'=0.
\end{cases}\tag{2\ensuremath{_{q}'}}\label{equation2q'}
\end{equation}
Define the finite sets
\begin{align*}
 & \mathcal{P}_{q}:=\{(n_{k})_{k=0}^{q-1}\in\mathbb{Z}_{\geq0}^{q}:\text{there is a solution }(a_{1},a_{2},a_{3},b_{1},b_{2},b_{3},a_{1}',a_{2}',a_{3}',b_{1}',b_{2}',b_{3}')\text{ of (\ref{equation2q'}) in }(1/q)\mathbb{Z}/\mathbb{Z}\\
 & \text{such that the number of }k/q\text{ in }\{\pm a_{1},\pm a_{2},\pm a_{3},\pm b_{1},\pm b_{2},\pm b_{3},\pm a_{1}',\pm a_{2}',\pm a_{3}',\pm b_{1}',\pm b_{2}',\pm b_{3}'\}\text{ is }n_{k}\}.
\end{align*}
In practice, we only need to calculate $\mathcal{P}_{q}$ for $q\leq13$.
For $q=2$ and $3$, we have
\[
\mathcal{P}_{2}=\{(20,4),(12,12),(4,20)\}
\]
and
\[
\mathcal{P}_{3}=\{(24,0,0),(16,4,4),(14,5,5),(12,6,6),(10,7,7),(8,8,8),(6,9,9),(4,10,10),(0,12,12)\}.
\]

\addtocounter{equation}{1}
\begin{lyxalgorithm}
\label{algorithm23}For a given relation of the form $G$, we check
that whether there are any possible $q$-types of $G$ that belong
to $\mathcal{P}_{q}$. If not, then we conclude that the relation
(\ref{equation1}) cannot be of the form $G$.
\end{lyxalgorithm}

Let us give some examples to illustrate Algorithm \ref{algorithm23}.

Suppose that  $G=(R_{23}:R_{3})$. By Lemma \ref{lemma12},
\[
G=\pm(\zeta_{6}+\zeta_{6}^{5}+\zeta_{23}+\cdots+\zeta_{23}^{22}).
\]
The $2$-type of $G$ is a rotation of $(22,2)$, which does not belong
to $\mathcal{P}_{2}$. Therefore, the relation (\ref{equation1})
cannot be of the form $G$.

Suppose that $G=(R_{19}:R_{3})+2R_{2}$. By Lemma \ref{lemma11} and
\ref{lemma12},
\[
G=\pm(\zeta_{6}+\zeta_{6}^{5}+\zeta_{19}+\cdots+\zeta_{19}^{22})+\xi+(-\xi)+\overline{\xi}+(-\overline{\xi}),
\]
where $\xi\in\mu_{114}$. The $2$-type of $(R_{19}:R_{3})$ is a
rotation of $(18,2)$, and the $2$-type of $R_{2}$ is $(1,1)$,
so the $2$-type of $G$ must be a rotation of $(20,4)$, which belongs
to $\mathcal{P}_{2}$. Therefore, Algorithm \ref{algorithm23} fails
for $q=2$. The $3$-type of $(R_{19}:R_{3})$ is $(18,1,1)$, and
the $3$-type of $R_{2}$ is a rotation of $(2,0,0)$, so the $3$-type
of $G$ must be $(22,1,1)$ or $(18,3,3)$, which does not belong
to $\mathcal{P}_{3}$. Therefore, the relation (\ref{equation1})
cannot be of the form $G$.

Suppose that $G=(R_{17}:R_{5},4R_{3})$. By Lemma \ref{lemma12},
\[
G=\zeta_{10}+\zeta_{10}^{3}+\zeta_{10}^{7}+\zeta_{10}^{9}+\cdots+(\zeta_{6}+\zeta_{6}^{5})\zeta_{17}^{j}+\cdots+(\zeta_{6}+\zeta_{6}^{5})\zeta_{17}^{k}+\cdots+(\zeta_{6}+\zeta_{6}^{5})\zeta_{17}^{17-k}+\cdots+(\zeta_{6}+\zeta_{6}^{5})\zeta_{17}^{17-j}+\cdots,
\]
for some $1\leq j<k\leq8$. The $2$-type of $G$ is $(12,12)$, which
belongs to $\mathcal{P}_{2}$. The $3$-type of $G$ is $(16,4,4)$,
which belongs to $\mathcal{P}_{3}$. Therefore, Algorithm \ref{algorithm23}
fails for $q=2$ and $3$. The $5$-type of $G$ is $(20,1,1,1,1)$,
which does not belong to $\mathcal{P}_{5}$. Therefore, the relation
(\ref{equation1}) cannot be of the form $G$.

It turns out that, by taking $q=2$, $3$, or $5$, Algorithm \ref{algorithm23}
works for any $G$ listed in Lemma \ref{lemma18}. Therefore, there
are no solutions of (\ref{equation1}) and (\ref{equation2}) such
that (\ref{equation1}) is of the form $G$.

The careful readers may have noticed that in Table \ref{table2},
there are one solution of order $34$ and four solutions of order
$102$. Although the prime $17$ is involved, their associated relations
(\ref{equation1}) can only be decomposed as $n_{2}R_{2}+n_{3}R_{3}$
for some $n_{2}$ and $n_{3}$, which will be discussed in Section
\ref{section3.4}.

\subsection{\label{section3.3}\texorpdfstring{$\boldsymbol{p=11,13}$}{$p=11,13$}}

The number of relations $G=0$ with $w(G)=24$ and $p=11,13$ is much
larger than the number of relations $G=0$ with $w(G)=24$ and $p=17,19,23$.
Since our purpose is finding the solutions of (\ref{equation1}) and
(\ref{equation2}) rather than classifying the vanishing relations,
we can first exclude some cases with certain shapes.
\begin{lem}
\label{lemma24}$ $
\begin{itemize}
\item If at least $7$ terms of $(n_{k})\in\mathcal{P}_{11}$ are $1$,
then $(n_{k})$ must be a Galois conjugate of $(4,6,1,\cdots,1,6)$.
\item If at least $9$ terms of $(n_{k})\in\mathcal{P}_{13}$ are $1$,
then $(n_{k})$ must be a Galois conjugate of $(6,4,1,\cdots,1,4)$,
$(4,5,1,\cdots,1,5)$, or $(2,6,1,\cdots,1,6)$.
\end{itemize}
\end{lem}

By Lemma \ref{lemma10}, if $p=13$, or $p=11$ and the longest minimal
relation occurring in $G$ has weight at most $21$, then $G$ must
be of the form (\ref{equation7}). If $G$ is not of this form, then
$G=G_{0}$ or $G_{0}+R_{2}$, where $G_{0}$ is minimal and stable
under complex conjugation. By Lemma \ref{lemma8}, $G_{0}$ can be
written as $\xi\sum_{i=0}^{10}f_{i}\zeta_{11}^{i}$, where $\xi\in\mu_{\infty}$
and each $f_{i}$ is a sum of at least two elements of $\mu_{210}$.
Since $[\mathbb{Q}(\zeta_{2310}):\mathbb{Q}(\zeta_{210})]=10$, we
know that $f_{i}-f_{j}=0$ for any $0\leq i<j\leq10$.

If $w(G_{0})=22$, then $w(f_{i}-f_{j})=4$ for any $0\leq i<j\leq10$.
But since the only relation of weight $4$ is $2R_{2}$, all $f_{i}$
must be the same, $G_{0}$ cannot be minimal. In general, if $f_{i}$
is a sum of two elements of $\mu_{\infty}$, then $\xi f_{i}=\xi f_{11-i}=\overline{\xi f_{i}}$.

If $w(G_{0})=24$ and the $11$-type of $G_{0}$ is $(4,2,\cdots,2)$,
then $w(f_{0}-f_{1})=6$. If $f_{0}-f_{1}=3R_{2}$ or $2R_{3}$, then
$G_{0}$ cannot be minimal. If $f_{0}-f_{1}=(R_{5}:R_{3})$, then
since both $\xi f_{0}$ and $\xi f_{1}$ are stable under complex
conjugation, by Lemma \ref{lemma12}, $G$ must be
\begin{equation}
\pm\left(\zeta_{5}+\zeta_{5}^{2}+\zeta_{5}^{3}+\zeta_{5}^{4}+\sum_{i=1}^{10}(\zeta_{3}+\zeta_{3}^{2})\zeta_{11}^{i}\right)\label{equation8}
\end{equation}
or a Galois conjugate of
\begin{equation}
\pm\left(\zeta_{3}+\zeta_{3}^{2}+\zeta_{10}+\zeta_{10}^{9}+\sum_{i=1}^{10}(\zeta_{5}+\zeta_{5}^{4})\zeta_{11}^{i}\right).\label{equation9}
\end{equation}
The $2$-type of $G$ is a rotation of $(24,0)$ or $(22,2)$, which
does not belong to $\mathcal{P}_{2}$. Therefore, the relation (\ref{equation1})
cannot be of the form $G$.

If $w(G_{0})=24$ and the $11$-type of $G_{0}$ is a Galois conjugate
of $(2,3,2,\cdots,2,3)$, then $w(f_{0}-f_{1})=5$. If $f_{0}-f_{1}=R_{2}+R_{3}$,
then $f_{0}-f_{10}=\overline{f_{0}-f_{1}}=R_{2}+R_{3}$, $G_{0}$
cannot be minimal. If $f_{0}-f_{1}=R_{5}$, since $\xi f_{0}$ is
stable under complex conjugation, $\xi f_{1}$ is forced to be stable
under complex conjugation, $G$ must be a Galois conjugate of
\begin{equation}
\pm\left((\zeta_{10}+\zeta_{10}^{9}-1)(\zeta_{11}+\zeta_{11}^{10})+\sum_{i\neq1,10}(\zeta_{5}+\zeta_{5}^{4})\zeta_{11}^{i}\right).\label{equation10}
\end{equation}
The $2$-type of $G$ is a rotation of $(18,6)$, which does not belong
to $\mathcal{P}_{2}$. Therefore, the relation (\ref{equation1})
cannot be of the form $G$.

Suppose that (\ref{equation1}) is of the form (\ref{equation7}),
and some of the $H_{i}$ and $G_{i}$ has weight greater than $12$,
then $G=(R_{13}:H_{1})$, $(R_{11}:H_{1})$, $(R_{11}:H_{1})+R_{2}$,
or $R_{11}+G_{1}$. For each case except for $R_{11}+(R_{11}:R_{3})$,
the $p$-type of $G$ is $(25-p,1,\cdots,1)$, which does not belong
to $\mathcal{P}_{p}$ by Lemma \ref{lemma24}. Since $R_{11}+(R_{11}:R_{3})=(R_{11}:R_{3})+R_{11}$,
we can assume that none of the $H_{i}$ and $G_{i}$ has weight greater
than $12$, so that Lemma \ref{lemma13} is enough for us.

(\ref{equation1}) cannot be of the form $(R_{p}:H_{1},\cdots,H_{j})+R_{2}$
because the $2$-components of $(R_{p}:H_{1},\cdots,H_{j})$ belong
to $\mu_{2}$, but the $2$-components of $R_{2}=\zeta_{4}+\zeta_{4}^{3}$
belong to $\mu_{4}$.

If (\ref{equation1}) is of the form $(R_{p}:H_{1},\cdots,H_{j})+\sum_{q<p}n_{q}R_{q}$,
then by Lemma \ref{lemma11}, $(R_{p}:H_{1},\cdots,H_{j})$ is stable
under complex conjugation. By Lemma \ref{lemma12}, among $H_{i}$,
all but possibly one entries of Lemma \ref{lemma13} must occur even
times, and only $R_{3}$, $R_{5}$, $(R_{5}:2R_{3})$, $(R_{5}:4R_{3})$,
$R_{7}$, $(R_{7}:2R_{3})$, $(R_{7}:4R_{3})$, or $R_{11}$ can be
the exception.

If (\ref{equation1}) is of the form $(R_{13}:H_{1},\cdots,H_{j})+\sum_{q<13}n_{q}R_{q}+G_{1}$,
where $n_{q}\in\{0,1\}$ for any $q$, then by Lemma \ref{lemma20},
the $13$-components of $\sum_{q<13}n_{q}R_{q}+G_{1}$ must be trivial.
By Lemma \ref{lemma24}, if $j\leq3$, then $\{H_{1},\cdots,H_{j}\}$
must contain $2R_{5}$, $2(R_{5}:R_{3})$, $2(R_{5}:2R_{3})$, or
$2R_{7}$.

If (\ref{equation1}) is of the form $(R_{11}:H_{1},\cdots,H_{j})+\sum_{q<11}n_{q}R_{q}+G_{1}$,
where $n_{q}\in\{0,1\}$ for any $q$, and $G_{1}\neq R_{11}$ or
$(R_{11}:R_{3})$, then by Lemma \ref{lemma20}, the $11$-components
of $\sum_{q<11}n_{q}R_{q}+G_{1}$ must be trivial. By Lemma \ref{lemma24},
if $j\leq3$, then $\{H_{1},\cdots,H_{j}\}$ must contain $2(R_{5}:2R_{3})$
or $2R_{7}$.

\addtocounter{table}{-1}
\begin{lem}
\label{lemma25}After these eliminations, the relations $G=0$ with
$w(G)=24$ and $p=11,13$ are as follows:

\begin{longtable}{|c|c|c|}
\hline 
$R_{13}+R_{7}+2R_{2}$ & $R_{13}+(R_{5}:2R_{3})+2R_{2}$ & $R_{13}+R_{5}+2R_{3}$\tabularnewline
\hline 
$R_{13}+R_{5}+3R_{2}$ & $R_{13}+3R_{3}+R_{2}$ & $\boldsymbol{R_{13}+R_{3}+4R_{2}^{*}}$\tabularnewline
\hline 
$(R_{13}:R_{3})+(R_{5}:R_{3})+2R_{2}$ & $\boldsymbol{(R_{13}:R_{3})+2R_{5}^{*}}$ & $\boldsymbol{(R_{13}:R_{3})+2R_{3}+2R_{2}^{*}}$\tabularnewline
\hline 
$\boldsymbol{(R_{13}:R_{3})+5R_{2}^{*}}$ & $(R_{13}:2R_{3})+R_{5}+2R_{2}$ & $\boldsymbol{(R_{13}:2R_{3})+3R_{3}^{*}}$\tabularnewline
\hline 
$(R_{13}:2R_{3})+R_{3}+3R_{2}$ & $(R_{13}:R_{5})+2R_{3}+R_{2}$ & $\boldsymbol{(R_{13}:R_{5})+4R_{2}^{*}}$\tabularnewline
\hline 
$\boldsymbol{(R_{13}:3R_{3})+2R_{3}+R_{2}^{*}}$ & $\boldsymbol{(R_{13}:3R_{3})+4R_{2}^{*}}$ & $(R_{13}:4R_{3})+R_{7}$\tabularnewline
\hline 
$(R_{13}:4R_{3})+(R_{5}:2R_{3})$ & $(R_{13}:4R_{3})+R_{5}+R_{2}$ & $(R_{13}:4R_{3})+R_{3}+2R_{2}$\tabularnewline
\hline 
$(R_{13}:R_{7})+2R_{3}$ & $(R_{13}:R_{7})+3R_{2}$ & $(R_{13}:(R_{5}:2R_{3}))+2R_{3}$\tabularnewline
\hline 
$(R_{13}:(R_{5}:2R_{3}))+3R_{2}$ & $(R_{13}:R_{5},2R_{3})+2R_{3}$ & $(R_{13}:R_{5},2R_{3})+3R_{2}$\tabularnewline
\hline 
$(R_{13}:5R_{3})+(R_{5}:R_{3})$ & $(R_{13}:5R_{3})+2R_{3}$ & $(R_{13}:5R_{3})+3R_{2}$\tabularnewline
\hline 
$(R_{13}:2R_{5})+R_{5}$ & $(R_{13}:2R_{5})+R_{3}+R_{2}$ & $(R_{13}:6R_{3})+R_{5}$\tabularnewline
\hline 
$(R_{13}:6R_{3})+R_{3}+R_{2}$ & $(R_{13}:(R_{7}:2R_{3}))+2R_{2}$ & $(R_{13}:(R_{5}:4R_{3}))+2R_{2}$\tabularnewline
\hline 
$(R_{13}:R_{7},2R_{3})+2R_{2}$ & $(R_{13}:(R_{5}:2R_{3}),2R_{3})+2R_{2}$ & $(R_{13}:2R_{5},R_{3})+2R_{2}$\tabularnewline
\hline 
$(R_{13}:R_{5},4R_{3})+2R_{2}$ & $(R_{13}:7R_{3})+2R_{2}$ & $(R_{13}:2(R_{5}:R_{3}))+R_{3}$\tabularnewline
\hline 
$(R_{13}:2R_{5},2R_{3})+R_{3}$ & $(R_{13}:8R_{3})+R_{3}$ & $(R_{13}:(R_{7}:2R_{3}),4R_{3})$\tabularnewline
\hline 
$(R_{13}:(R_{5}:4R_{3}),4R_{3})$ & $(R_{13}:2R_{7},R_{3})$ & $(R_{13}:2(R_{5}:2R_{3}),R_{3})$\tabularnewline
\hline 
$(R_{13}:R_{7},2R_{5})$ & $(R_{13}:(R_{5}:2R_{3}),2R_{5})$ & $(R_{13}:R_{7},6R_{3})$\tabularnewline
\hline 
$(R_{13}:(R_{5}:2R_{3}),6R_{3})$ & $(R_{13}:2(R_{5}:R_{3}),R_{5})$ & $(R_{13}:2(R_{5}:R_{3}),3R_{3})$\tabularnewline
\hline 
$(R_{13}:3R_{5},2R_{3})$ & $(R_{13}:2R_{5},5R_{3})$ & $(R_{13}:R_{5},8R_{3})$\tabularnewline
\hline 
$(R_{13}:11R_{3})$ &  & \tabularnewline
\hline 
$2R_{11}+R_{2}$ & $R_{11}+(R_{7}:2R_{3})+2R_{2}$ & $R_{11}+(R_{5}:4R_{3})+2R_{2}$\tabularnewline
\hline 
$R_{11}+(R_{5}:2R_{3})+(R_{5}:R_{3})$ & $R_{11}+R_{7}+2R_{3}$ & $R_{11}+(R_{5}:2R_{3})+2R_{3}$\tabularnewline
\hline 
$R_{11}+R_{7}+3R_{2}$ & $R_{11}+(R_{5}:2R_{3})+3R_{2}$ & $R_{11}+(R_{5}:R_{3})+R_{3}+2R_{2}$\tabularnewline
\hline 
$\boldsymbol{R_{11}+2R_{5}+R_{3}^{*}}$ & $R_{11}+R_{5}+2R_{3}+R_{2}$ & $\boldsymbol{R_{11}+R_{5}+4R_{2}^{*}}$\tabularnewline
\hline 
$\boldsymbol{R_{11}+3R_{3}+2R_{2}^{*}}$ & $\boldsymbol{R_{11}+R_{3}+5R_{2}^{*}}$ & $2(R_{11}:R_{3})$\tabularnewline
\hline 
$(R_{11}:R_{3})+(R_{7}:R_{3})+2R_{2}$ & $(R_{11}:R_{3})+(R_{5}:3R_{3})+2R_{2}$ & $(R_{11}:R_{3})+2(R_{5}:R_{3})$\tabularnewline
\hline 
$(R_{11}:R_{3})+(R_{5}:R_{3})+2R_{3}$ & $(R_{11}:R_{3})+(R_{5}:R_{3})+3R_{2}$ & $\boldsymbol{(R_{11}:R_{3})+2R_{5}+R_{2}^{*}}$\tabularnewline
\hline 
$(R_{11}:R_{3})+R_{5}+R_{3}+2R_{2}$ & $\boldsymbol{(R_{11}:R_{3})+4R_{3}^{*}}$ & $\boldsymbol{(R_{11}:R_{3})+2R_{3}+3R_{2}^{*}}$\tabularnewline
\hline 
$\boldsymbol{(R_{11}:R_{3})+6R_{2}^{*}}$ & $(R_{11}:2R_{3})+R_{11}$ & $(R_{11}:2R_{3})+R_{7}+2R_{2}$\tabularnewline
\hline 
$(R_{11}:2R_{3})+(R_{5}:2R_{3})+2R_{2}$ & $\boldsymbol{(R_{11}:2R_{3})+R_{5}+2R_{3}^{*}}$ & $(R_{11}:2R_{3})+R_{5}+3R_{2}$\tabularnewline
\hline 
$\boldsymbol{(R_{11}:2R_{3})+3R_{3}+R_{2}^{*}}$ & $\boldsymbol{(R_{11}:2R_{3})+R_{3}+4R_{2}^{*}}$ & $(R_{11}:R_{5})+(R_{5}:R_{3})+2R_{2}$\tabularnewline
\hline 
$\boldsymbol{(R_{11}:R_{5})+2R_{5}^{*}}$ & $\boldsymbol{(R_{11}:R_{5})+2R_{3}+2R_{2}^{*}}$ & $\boldsymbol{(R_{11}:R_{5})+5R_{2}^{*}}$\tabularnewline
\hline 
$(R_{11}:3R_{3})+(R_{5}:R_{3})+2R_{2}$ & $\boldsymbol{(R_{11}:3R_{3})+2R_{5}^{*}}$ & $\boldsymbol{(R_{11}:3R_{3})+2R_{3}+2R_{2}^{*}}$\tabularnewline
\hline 
$\boldsymbol{(R_{11}:3R_{3})+5R_{2}^{*}}$ & $(R_{11}:4R_{3})+(R_{7}:2R_{3})$ & $(R_{11}:4R_{3})+(R_{5}:4R_{3})$\tabularnewline
\hline 
$(R_{11}:4R_{3})+R_{7}+R_{2}$ & $(R_{11}:4R_{3})+(R_{5}:2R_{3})+R_{2}$ & $(R_{11}:4R_{3})+(R_{5}:R_{3})+R_{3}$\tabularnewline
\hline 
$(R_{11}:4R_{3})+R_{5}+2R_{2}$ & $(R_{11}:4R_{3})+3R_{3}$ & $(R_{11}:4R_{3})+R_{3}+3R_{2}$\tabularnewline
\hline 
$(R_{11}:R_{7})+2R_{3}+R_{2}$ & $\boldsymbol{(R_{11}:R_{7})+4R_{2}^{*}}$ & $(R_{11}:(R_{5}:2R_{3}))+2R_{3}+R_{2}$\tabularnewline
\hline 
$\boldsymbol{(R_{11}:(R_{5}:2R_{3}))+4R_{2}^{*}}$ & $\boldsymbol{(R_{11}:R_{5},2R_{3})+2R_{3}+R_{2}^{*}}$ & $\boldsymbol{(R_{11}:R_{5},2R_{3})+4R_{2}^{*}}$\tabularnewline
\hline 
$(R_{11}:5R_{3})+(R_{7}:R_{3})$ & $(R_{11}:5R_{3})+(R_{5}:3R_{3})$ & $(R_{11}:5R_{3})+(R_{5}:R_{3})+R_{2}$\tabularnewline
\hline 
$(R_{11}:5R_{3})+R_{5}+R_{3}$ & $\boldsymbol{(R_{11}:5R_{3})+2R_{3}+R_{2}^{*}}$ & $\boldsymbol{(R_{11}:5R_{3})+4R_{2}^{*}}$\tabularnewline
\hline 
$(R_{11}:2R_{5})+R_{3}+2R_{2}$ & $(R_{11}:R_{5},3R_{3})+(R_{5}:2R_{3})$ & $(R_{11}:6R_{3})+R_{7}$\tabularnewline
\hline 
$(R_{11}:6R_{3})+(R_{5}:2R_{3})$ & $(R_{11}:6R_{3})+R_{5}+R_{2}$ & $(R_{11}:6R_{3})+R_{3}+2R_{2}$\tabularnewline
\hline 
$(R_{11}:(R_{7}:2R_{3}))+2R_{3}$ & $(R_{11}:(R_{7}:2R_{3}))+3R_{2}$ & $(R_{11}:(R_{5}:4R_{3}))+2R_{3}$\tabularnewline
\hline 
$(R_{11}:(R_{5}:4R_{3}))+3R_{2}$ & $(R_{11}:R_{7},2R_{3})+2R_{3}$ & $(R_{11}:R_{7},2R_{3})+3R_{2}$\tabularnewline
\hline 
$\boldsymbol{(R_{11}:(R_{5}:2R_{3}),2R_{3})+2R_{3}^{*}}$ & $(R_{11}:(R_{5}:2R_{3}),2R_{3})+3R_{2}$ & $(R_{11}:(R_{5}:R_{3}),3R_{3})+(R_{5}:R_{3})$\tabularnewline
\hline 
$(R_{11}:2R_{5},R_{3})+2R_{3}$ & $(R_{11}:2R_{5},R_{3})+3R_{2}$ & $(R_{11}:R_{5},4R_{3})+(R_{5}:R_{3})$\tabularnewline
\hline 
$\boldsymbol{(R_{11}:R_{5},4R_{3})+2R_{3}^{*}}$ & $(R_{11}:R_{5},4R_{3})+3R_{2}$ & $(R_{11}:7R_{3})+(R_{5}:R_{3})$\tabularnewline
\hline 
$\boldsymbol{(R_{11}:7R_{3})+2R_{3}^{*}}$ & $(R_{11}:7R_{3})+3R_{2}$ & $(R_{11}:2R_{5},2R_{3})+R_{5}$\tabularnewline
\hline 
$(R_{11}:2R_{5},2R_{3})+R_{3}+R_{2}$ & $(R_{11}:8R_{3})+R_{5}$ & $\boldsymbol{(R_{11}:8R_{3})+R_{3}+R_{2}^{*}}$\tabularnewline
\hline 
$(R_{11}:(R_{7}:4R_{3}))+2R_{2}$ & $(R_{11}:(R_{7}:2R_{3}),2R_{3})+2R_{2}$ & $(R_{11}:(R_{5}:4R_{3}),2R_{3})+2R_{2}$\tabularnewline
\hline 
$(R_{11}:R_{7},4R_{3})+2R_{2}$ & $\boldsymbol{(R_{11}:(R_{5}:2R_{3}),4R_{3})+2R_{2}^{*}}$ & $(R_{11}:2(R_{5}:R_{3}),R_{3})+2R_{2}$\tabularnewline
\hline 
$(R_{11}:3R_{5})+2R_{2}$ & $(R_{11}:2R_{5},3R_{3})+2R_{2}$ & $(R_{11}:R_{5},6R_{3})+2R_{2}$\tabularnewline
\hline 
$\boldsymbol{(R_{11}:9R_{3})+2R_{2}^{*}}$ & $(R_{11}:2R_{7})+R_{3}$ & $\boldsymbol{(R_{11}:2(R_{5}:2R_{3}))+R_{3}^{*}}$\tabularnewline
\hline 
$(R_{11}:2(R_{5}:R_{3}),2R_{3})+R_{3}$ & $(R_{11}:2R_{5},4R_{3})+R_{3}$ & $(R_{11}:10R_{3})+R_{3}$\tabularnewline
\hline 
$(R_{11}:(R_{7}:4R_{3}),4R_{3})$ & $(R_{11}:(R_{7}:2R_{3}),6R_{3})$ & $(R_{11}:(R_{5}:4R_{3}),6R_{3})$\tabularnewline
\hline 
$(R_{11}:2R_{7},R_{5})$ & $(R_{11}:2(R_{5}:2R_{3}),R_{5})$ & $(R_{11}:2R_{7},3R_{3})$\tabularnewline
\hline 
$(R_{11}:2(R_{5}:2R_{3}),3R_{3})$ & $(R_{11}:R_{7},2R_{5},2R_{3})$ & $(R_{11}:(R_{5}:2R_{3}),2R_{5},2R_{3})$\tabularnewline
\hline 
$(R_{11}:R_{7},8R_{3})$ & $(R_{11}:(R_{5}:2R_{3}),8R_{3})$ & $(R_{11}:2(R_{5}:R_{3}),R_{5},2R_{3})$\tabularnewline
\hline 
$(R_{11}:2(R_{5}:R_{3}),5R_{3})$ & $(R_{11}:4R_{5},R_{3})$ & $\boldsymbol{(R_{11}:3R_{5},4R_{3})^{*}}$\tabularnewline
\hline 
$(R_{11}:2R_{5},7R_{3})$ &  & \tabularnewline
\hline 
\end{longtable}
\end{lem}

By Lemma \ref{lemma20}, $(R_{11}:R_{5},3R_{3})+(R_{5}:2R_{3})$ and
$(R_{11}:(R_{5}:R_{3}),3R_{3})+(R_{5}:R_{3})$ cannot be stable under
complex conjugation, so we can ignore them. For $R_{11}+(R_{5}:2R_{3})+(R_{5}:R_{3})$,
we can apply Algorithm \ref{algorithm23} by taking $q=3$ without
knowing whether $(R_{5}:2R_{3})$ and $(R_{5}:R_{3})$ are stable
under complex conjugation.

Suppose $G=2R_{11}+R_{2}$. If $R_{2}=1+(-1)$, then $G$ can be decomposed
as $12R_{2}$, which will be discussed in Section \ref{section3.4}.
Otherwise, $G=\xi R_{11}+\overline{\xi}R_{11}+\zeta_{4}+\zeta_{4}^{3}$.
From the relations (\ref{equation2}), we see that either $\xi\in\mu_{44}$
or $\xi^{3}\in\mu_{44}$. If $\xi\in\mu_{44}$, then again $G$ can
be decomposed as $12R_{2}$. Otherwise, we can apply Algorithm \ref{algorithm23}
by taking $q=3$.

Suppose $G=\xi_{1}(R_{11}:R_{3})+\xi_{2}(R_{11}:R_{3})$. If $\xi_{1}(R_{11}:R_{3})$
and $\xi_{2}(R_{11}:R_{3})$ are not conjugate to each other, then
by Lemma \ref{lemma20}, $\xi_{1},\xi_{2}\in\mu_{132}$. Otherwise,
$\xi_{1}=\overline{\xi_{2}}$. Without loss of generality, let $A_{1}=\xi_{1}\zeta_{6}$,
then from the relations (\ref{equation2}), either $A_{1}B_{2}A_{1}'B_{2}'=1$
or $A_{1}B_{3}\overline{A_{1}'B_{3}'}=1$ does not contain $\xi_{1}\zeta_{6}^{5}$.
Therefore, the $3$-components cannot be canceled, which implies that
either $\xi_{1}^{2}\in\mu_{66}$ or $\xi_{1}^{4}\in\mu_{66}$. In
any case, we can apply Algorithm \ref{algorithm23} by taking $q=3$.

Suppose $G=\xi_{1}(R_{11}:R_{3})+\xi_{2}(R_{5}:R_{3})+\xi_{3}(R_{5}:R_{3})$.
If $\xi_{2}(R_{5}:R_{3})$ and $\xi_{3}(R_{5}:R_{3})$ are not conjugate
to each other, then by the following lemma, $\xi_{1},\xi_{2},\xi_{3}\in\mu_{660}$.
Otherwise, $\xi_{1}=\pm1$ and $\xi_{2}=\overline{\xi_{3}}$. From
the relations (\ref{equation2}), either $\xi_{2}^{k}\in\mu_{330}$
for some $1\leq k\leq6$, or $\xi_{2}$ can be canceled, which means
that $\xi_{2}$ can be taken arbitrarily. In any case, we can apply
Algorithm \ref{algorithm23} by taking $q=11$.
\begin{lem}
\label{lemma26}Let $G=\xi_{1}G_{1}+\xi_{2}G_{2}+\xi_{3}G_{3}$ be
a relation which is stable under complex conjugation, where $\xi_{i}\in\mu_{\infty}$
and $G_{i}$ are minimal relations. Assume that all the terms of $G_{i}$
belong to $\mu_{n}$. If any two of $\xi_{i}G_{i}$ are not conjugate
to each other, then $\xi_{i}\in\mu_{2n}$.
\end{lem}

\begin{proof}
If $\xi_{1}\notin\mu_{2n}$, then any two terms of $\xi_{1}G_{1}$
cannot be conjugate to each other. Since $\xi_{2}G_{2}$ is minimal,
and $\xi_{1}G_{1}$ is not conjugate to $\xi_{2}G_{2}$, some term
of $\xi_{1}G_{1}$ must be conjugate to some term of $\xi_{3}G_{3}$.
Therefore, $\xi_{1}\xi_{3}\in\mu_{n}$. By the same reasoning, $\xi_{1}\xi_{2}\in\mu_{n}$.
If $\xi_{2}\in\mu_{2n}$ or $\xi_{3}\in\mu_{2n}$, then we are done.
Otherwise, we must have $\xi_{2}\xi_{3}\in\mu_{n}$, then $\xi_{1}^{2}\xi_{2}^{2}\xi_{3}^{2}\in\mu_{n}$,
$\xi_{1}\xi_{2}\xi_{3}\in\mu_{2n}$, and finally $\xi_{1},\xi_{2},\xi_{3}\in\mu_{2n}$.
\end{proof}
By Lemma \ref{lemma11} and the same reasoning of Lemma \ref{lemma21},
any other $G$ listed in Lemma \ref{lemma25} satisfy Condition \ref{condition19}.
It turns out that, for those unmarked entries, we can apply Algorithm
\ref{algorithm23} by taking $q=2$, $3$, $5$, $7$, or $11$. However,
for the remaining cases, we need a new algorithm.
\begin{lyxalgorithm}
\label{algorithm27}Given a relation $G=G_{0}+n_{2}R_{2}+n_{3}R_{3}+n_{5}R_{5}$
of weight $24$, where $G_{0}$ is stable under complex conjugation
but not necessarily minimal. Let $X=\{X_{j}\}_{j=1}^{n}$ be the partition
of the terms of $G_{0}$ into conjugate pairs. $X$ is allowed to
contain at most one $\{1\}$ and at most one $\{-1\}$. For any $1\leq i\leq19$,
let $Y_{i}$ be a fixed element (for example, the $(i,2)$-entry of
Lemma \ref{lemma16}) in the orbit $O_{i}$. 
\begin{itemize}
\item Suppose $n=6$. For any $1\leq i\leq19$, we assign each element of
$Y_{i}$ a term of $G_{0}$ such that any two of them do not come
from the same $X_{j}$. Then we solve the equations (\ref{equation2'})
as we have demonstrated after Lemma \ref{lemma16} to get a full set
of $U_{24}$. There will be zero, one, or two free variables depending
on whether $1\leq i\leq9$, $10\leq i\leq18$, or $i=19$. Then we
check whether $G_{0}\subseteq U_{24}$ and $U_{24}\backslash G_{0}$
is possible to be of the form $n_{2}R_{2}+n_{3}R_{3}+n_{5}R_{5}$.
If yes, then we use (\ref{equation3}) to get the solutions in $\mathcal{S}$.
\item Suppose $n=7$. It suffices to take $i=1$, $3$, $5$, $7$, $11$,
and $14$ (see Lemma \ref{lemma17}).
\item Suppose $n=8$. It suffices to take $i=1$, $2$, $6$, $7$, and
$15$.
\item Suppose $n\geq9$. It suffices to take $i=1$ and select the elements
of $Y_{1}$ from the subset $\cup_{j=1}^{9}X_{j}\subseteq G_{0}$.
\end{itemize}
\end{lyxalgorithm}

Applying Algorithm \ref{algorithm27} to those marked entries listed
in Lemma \ref{lemma25}, we will get $21$ zero-parameter solutions:
ten have order $66$, seven have order $78$, one has order $110$,
and three have order $132$.

To illustrate Algorithm \ref{algorithm27}, let us give an example
to see how to get the solution
\[
[1/66,2/33,9/22,1/33,5/44,5/12]\in\mathcal{S}^{+}
\]
from the relation
\[
G=G_{0}+R_{3}+4R_{2}=(R_{11}:2R_{3})+R_{3}+4R_{2}=(\zeta_{6}+\zeta_{6}^{5})(\zeta_{11}+\zeta_{11}^{10})+\sum_{i\neq1,10}\zeta_{11}^{i}+R_{3}+4R_{2}.
\]
Now $n=7$, so we can take $i=14$. Let
\[
(a_{1},a_{2},b_{1},b_{3},a_{1}',a_{2}')=(2/11,5/66,0,7/11,3/11,49/66).
\]
By Lemma \ref{lemma16}, (\ref{equation2'}) has a solution if and
only if $a_{1}+a_{2}+2b_{3}-a_{1}'+a_{2}'=0$. Now this condition
is satisfied, so (\ref{equation2'}) has the solution
\begin{eqnarray*}
 &  & (a_{1},a_{2},a_{3},b_{1},b_{2},b_{3},a_{1}',a_{2}',a_{3}',b_{1}',b_{2}',b_{3}')\\
 & = & (2/11,5/66,2/33+x,0,6/11-x,7/11,3/11,49/66,20/33-x,1/3,x,6/11),
\end{eqnarray*}
where $x$ is a free variable. Now $G_{0}\subseteq U_{24}$ and the
image of $U_{24}\backslash G_{0}$ under $\textup{Arg}/(2\pi)$ is
\[
\{2/33+x,6/11-x,20/33-x,1/3,x,31/33-x,0,5/11+x,13/33+x,2/3,-x\}.
\]
It is not hard to see that $U_{24}\backslash G_{0}$ can be decomposed
as $R_{3}+4R_{2}$ if and only if $x=1/44$ or $23/44$. If we take
$x=23/44$, then by (\ref{equation3}), we get
\[
\left[\frac{37}{66},\frac{3}{22},\frac{8}{33},\frac{4}{33},\frac{5}{12},\frac{9}{44}\right]\overset{(23)(56)}{\thickapprox}\left[\frac{37}{66},\frac{8}{33},\frac{3}{22},\frac{4}{33},\frac{9}{44},\frac{5}{12}\right]\overset{\times25}{\thickapprox}\left[\frac{1}{66},\frac{2}{33},\frac{9}{22},\frac{1}{33},\frac{5}{44},\frac{5}{12}\right].
\]
If we take $x=1/44$, we will get the same solution.

If $n=6$ and $i=19$, then $U_{24}\backslash G_{0}$ may have two
free variables, which will bring additional difficulties. But fortunately,
this situation can always be avoided in practice.

\subsection{\label{section3.4}\texorpdfstring{$\boldsymbol{p=2,3}$}{$p=2,3$}}

By Lemma \ref{lemma9}, the relations $G=0$ with $w(G)=24$ and $p=2,3$
are $12R_{2}$, $9R_{2}+2R_{3}$, $6R_{2}+4R_{3}$, $3R_{2}+6R_{3}$,
and $8R_{3}$.
\begin{lyxalgorithm}
\label{algorithm28}Given a relation $G$ of weight $24$ that satisfies
Condition \ref{condition19}, then $G$ can be decomposed as
\begin{equation}
G=G_{0}+2G_{1}+\cdots+2G_{k}=G_{0}+\xi_{1}G_{1}+\overline{\xi_{1}G_{1}}+\cdots+\xi_{k}G_{k}+\overline{\xi_{k}G_{k}},\label{equation11}
\end{equation}
where $G_{0}$ is allowed to be empty and not necessarily minimal,
$G_{i}$ are minimal for $i\neq0$, and $\xi_{i}$ are unknowns to
be determined. Let $X=\{X_{j}\}_{j=1}^{12}$ be the partition of the
terms of $G$ into conjugate pairs. We assign each element of $U_{12}$
a term of $G$ such that any two of them do not come from the same
$X_{j}$. Then (\ref{equation2'}) becomes a system of six linear
equations in $k\leq6$ unknowns. Let $M$ be the associated matrix
of (\ref{equation2'}), $D$ the Smith normal form of $M$, and
\[
d=\begin{cases}
D_{kk} & \text{if }D_{kk}\neq0,\\
\infty & \text{if }D_{kk}=0.
\end{cases}
\]
Assume that all the terms of $G_{i}$ belong to $\mu_{n}$ for some
even $n$. Depending on the value of $n$, we choose another nonnegative
integer $m$.
\begin{itemize}
\item If $d>m$, then we solve the equations (\ref{equation2'}) and use
(\ref{equation3}) to get the solutions in $\mathcal{S}$.
\item If $d\leq m$, then all the terms of $G$ belong to $\mu_{nd}$. This
situation will be discussed in Section \ref{section3.6}.
\end{itemize}
\end{lyxalgorithm}

Applying Algorithm \ref{algorithm28} to $G=n_{2}R_{2}+n_{3}R_{3}$,
where $2n_{2}+3n_{3}=24$, by choosing $m=0$, we will get all the
positive-parameter solutions and $133$ zero-parameter solutions.
We note that none of the positive-parameter solutions comes from $G=8R_{3}$.

To illustrate Algorithm \ref{algorithm28}, let us give an example
to see how to get the solution
\[
[x,3x,1/2+4x,4x,12x,1/2+6x]\in\mathcal{S}^{+}
\]
from the relation $G=12R_{2}$. Let
\begin{eqnarray*}
 &  & (a_{1},a_{2},a_{3},b_{1},b_{2},b_{3},a_{1}',a_{2}',a_{3}',b_{1}',b_{2}',b_{3}')\\
 & = & (x_{1},x_{2},1/2+x_{2},x_{3},x_{4},x_{5},1/2-x_{1},1/2+x_{5},1/2+x_{3},x_{6},1/2-x_{4},1/2-x_{6}),
\end{eqnarray*}
then (\ref{equation2'}) becomes
\[
\begin{cases}
x_{1}+2x_{2}+x_{3}+x_{4}+x_{5}=0,\\
-x_{1}+x_{3}-x_{4}+x_{5}=0,\\
x_{2}+2x_{5}-x_{6}=0,\\
x_{2}+2x_{3}+x_{6}=0,\\
x_{2}-x_{3}+2x_{4}=1/2,\\
2x_{1}+x_{5}+x_{6}=0,
\end{cases}\text{or }\begin{pmatrix}1 & 2 & 1 & 1 & 1 & 0\\
-1 & 0 & 1 & -1 & 1 & 0\\
0 & 1 & 0 & 0 & 2 & -1\\
0 & 1 & 2 & 0 & 0 & 1\\
0 & 1 & -1 & 2 & 0 & 0\\
2 & 0 & 0 & 0 & 1 & 1
\end{pmatrix}\begin{pmatrix}x_{1}\\
x_{2}\\
x_{3}\\
x_{4}\\
x_{5}\\
x_{6}
\end{pmatrix}=\begin{pmatrix}0\\
0\\
0\\
0\\
1/2\\
0
\end{pmatrix}.
\]
Although we have seen how to solve such equations after Lemma \ref{lemma16},
let us do it again for the completeness of this example. By Theorem
\ref{theorem14},
\[
\begin{pmatrix}0 & 1 & -1 & 0 & 1 & 1\\
0 & 0 & 1 & 0 & 0 & 0\\
0 & 2 & -1 & 0 & 1 & 1\\
0 & 6 & -3 & -1 & 4 & 3\\
0 & 4 & -1 & -1 & 2 & 2\\
1 & 1 & -1 & -1 & 0 & 0
\end{pmatrix}^{-1}\begin{pmatrix}1 & 0 & 0 & 0 & 0 & 0\\
0 & 1 & 0 & 0 & 0 & 0\\
0 & 0 & 1 & 0 & 0 & 0\\
0 & 0 & 0 & 1 & 0 & 0\\
0 & 0 & 0 & 0 & 2 & 0\\
0 & 0 & 0 & 0 & 0 & 0
\end{pmatrix}\begin{pmatrix}1 & 0 & 0 & -1 & 0 & -1\\
0 & 1 & 0 & 4 & 5 & 8\\
0 & 0 & 1 & -3 & -5 & -6\\
0 & 0 & 0 & -3 & -6 & -7\\
0 & 0 & 0 & -1 & -1 & -2\\
0 & 0 & 0 & 2 & 3 & 4
\end{pmatrix}^{-1}\begin{pmatrix}x_{1}\\
x_{2}\\
x_{3}\\
x_{4}\\
x_{5}\\
x_{6}
\end{pmatrix}=\begin{pmatrix}0\\
0\\
0\\
0\\
1/2\\
0
\end{pmatrix}.
\]
Multiplying the inverse of the first matrix on both sides, we get
\[
\begin{pmatrix}1 & 0 & 0 & 0 & 0 & 0\\
0 & 1 & 0 & 0 & 0 & 0\\
0 & 0 & 1 & 0 & 0 & 0\\
0 & 0 & 0 & 1 & 0 & 0\\
0 & 0 & 0 & 0 & 2 & 0\\
0 & 0 & 0 & 0 & 0 & 0
\end{pmatrix}\begin{pmatrix}1 & 0 & 0 & -1 & 0 & -1\\
0 & 1 & 0 & 4 & 5 & 8\\
0 & 0 & 1 & -3 & -5 & -6\\
0 & 0 & 0 & -3 & -6 & -7\\
0 & 0 & 0 & -1 & -1 & -2\\
0 & 0 & 0 & 2 & 3 & 4
\end{pmatrix}^{-1}\begin{pmatrix}x_{1}\\
x_{2}\\
x_{3}\\
x_{4}\\
x_{5}\\
x_{6}
\end{pmatrix}=\begin{pmatrix}1/2\\
0\\
1/2\\
0\\
0\\
0
\end{pmatrix}.
\]
Note that the last entry of the matrix on the right-hand side is $0$,
so
\[
\begin{pmatrix}1 & 0 & 0 & -1 & 0 & -1\\
0 & 1 & 0 & 4 & 5 & 8\\
0 & 0 & 1 & -3 & -5 & -6\\
0 & 0 & 0 & -3 & -6 & -7\\
0 & 0 & 0 & -1 & -1 & -2\\
0 & 0 & 0 & 2 & 3 & 4
\end{pmatrix}^{-1}\begin{pmatrix}x_{1}\\
x_{2}\\
x_{3}\\
x_{4}\\
x_{5}\\
x_{6}
\end{pmatrix}=\begin{pmatrix}1/2\\
0\\
1/2\\
0\\
0\\
x
\end{pmatrix}\text{ or }\begin{pmatrix}1/2\\
0\\
1/2\\
0\\
1/2\\
x
\end{pmatrix},
\]
where $x$ is a free variable. Multiplying the inverse of the first
matrix on both sides, we get
\[
(x_{1},x_{2},x_{3},x_{4},x_{5},x_{6})=\begin{cases}
(1/2-x,8x,1/2-6x,-7x,-2x,4x),\\
(1/2-x,1/2+8x,-6x,-7x,1/2-2x,1/2+4x).
\end{cases}
\]
Substituting them into $U_{12}$, we get
\begin{eqnarray*}
 &  & (a_{1},a_{2},a_{3},b_{1},b_{2},b_{3},a_{1}',a_{2}',a_{3}',b_{1}',b_{2}',b_{3}')\\
 & = & \begin{cases}
(1/2-x,8x,1/2+8x,1/2-6x,-7x,-2x,x,1/2-2x,-6x,4x,1/2+7x,1/2-4x),\\
(1/2-x,1/2+8x,8x,-6x,-7x,1/2-2x,x,-2x,1/2-6x,1/2+4x,1/2+7x,-4x).
\end{cases}
\end{eqnarray*}
By (\ref{equation3}), we get
\[
\begin{cases}
[1/2+x,1/2+3x,1/2+4x,4x,12x,1/2+6x],\\{}
[x,3x,1/2+4x,4x,12x,1/2+6x].
\end{cases}
\]
Clearly, the second solution can be obtained by replacing $x$ with
$1/2+x$ in the first solution.

\subsection{\label{section3.5}\texorpdfstring{$\boldsymbol{p=5,7}$}{$p=5,7$}}

The purpose of this section to show that
\begin{lem}
\label{lemma29}If the associated relation (\ref{equation1}) of $[s_{1},s_{2},s_{3},t_{1},t_{2},t_{3}]\in\mathcal{S}$
has $p=5,7$, then $s_{i},t_{i}\in\mu_{n}$ for some $n\in N:=\{240,252,280,300,336,360,420\}$.
\end{lem}

As before, we first try to classify the relations $G=0$ with $w(G)=24$
and $p=5,7$. The following lemma shows that all the minimal relations
with $p=5$ have been given in Lemma \ref{lemma13}.
\begin{lem}
\label{lemma30}Suppose $G$ is a minimal relation with $p=5$. Then
$G=R_{5}$ or $(R_{5}:jR_{3})$ for some $j<5$.
\end{lem}

\begin{proof}
By Lemma \ref{lemma8}, $G$ can be written as a rotation of $\sum_{i=0}^{4}f_{i}\zeta_{5}^{i}$,
where each $f_{i}$ is a sum of elements of $\mu_{6}$. Since $[\mathbb{Q}(\zeta_{30}):\mathbb{Q}(\zeta_{6})]=4$,
we know that $f_{0}-f_{i}=0$ for any $0\leq i\leq4$. By another
rotation, we can assume that $f_{0}$ contains $1$, which belongs
to some $R_{2}$ or $R_{3}$ by Lemma \ref{lemma9}.

If $f_{0}$ contains $1$ but not $\zeta_{3}$ or $\zeta_{3}^{2}$,
then $f_{i}$ contains either $1$ or $\zeta_{6}+\zeta_{6}^{5}$.
By the minimality of $G$, $G=R_{5}$ or $(R_{5}:jR_{3})$, where
$j$ is the number of $f_{i}$ that contains $\zeta_{6}+\zeta_{6}^{5}$.

If $f_{0}$ contains $1+\zeta_{3}$, then $f_{i}$ contains either
$1+\zeta_{3}$ or $\zeta_{6}$. By the minimality of $G$, not all
$f_{i}$ contain $1+\zeta_{3}$, so $G=(R_{5}:jR_{3})$, where $j$
is the number of $f_{i}$ that contains $1+\zeta_{3}$.
\end{proof}
Therefore, the relations $G=0$ with $w(G)=24$ and $p=5$ are one-to-one
corresponding to the partitions of $24$ such that every summand belongs
to $\{2,3,5,6,7,8,9\}$ and some summand belongs to $\{5,6,7,8,9\}$.
However, the situation becomes much more complicated when $p=7$.
\begin{defn}
\label{definition31}Let $G=G_{1}+\cdots+G_{k}+2H_{1}+\cdots+2H_{j}$,
where $G_{i}$ and $H_{i}$ are minimal relations. We call $2H_{1}+\cdots+2H_{j}$
the even part of $G$.
\end{defn}

\begin{lem}
\label{lemma32}Suppose $G=0$ is a relation with $w(G)=24$ and $p=5,7$.
If the even part of $G$ is empty, $2R_{2}$, or $2R_{3}$, and $G$
is stable under complex conjugation, then all the terms of $G$ belong
to $\mu_{420}$.
\end{lem}

\begin{proof}
Suppose $G=G_{1}+\cdots+G_{k}+\sum_{q\leq7}n_{q}R_{q}$, where $G_{i}\neq R_{q}$
are minimal. By the assumptions, $k\leq3$ and any two of $G_{i}$
are not conjugate to each other. By Lemma \ref{lemma8}, \ref{lemma20},
and \ref{lemma26}, all the terms of $G_{1}+\cdots+G_{k}$ belong
to $\mu_{420}$, and consequently, all the terms of $G$ belong to
$\mu_{420}$.
\end{proof}
\addtocounter{table}{-1}
\begin{lem}
\label{lemma33}Suppose $G=0$ is a relation with $w(G)=24$ and $p=5,7$.
If the even part of $G$ is not empty, $2R_{2}$, or $2R_{3}$, and
$G$ can be decomposed as $G=G_{0}+n_{2}R_{2}+n_{3}R_{3}+n_{5}R_{5}$,
where $G_{0}$ is minimal, of weight at least $11$, and stable under
complex conjugation, then up to Galois conjugation, $G_{0}$ must
be one of the following:

\begin{longtable}{|c|c|c|}
\hline 
$(R_{7}:4R_{3})$ & $(R_{7}:5R_{3})$ & $(R_{7}:R_{5},2R_{3})$\tabularnewline
\hline 
$(R_{7}:(R_{5}:2R_{3}))$ & $(R_{7}:6R_{3})$ & $(R_{7}:2R_{5})$\tabularnewline
\hline 
$(R_{7}:R_{5},4R_{3})$ & $(R_{7}:2R_{5},R_{3})$ & $(R_{7}:(R_{5},2R_{3}),2R_{3})$\tabularnewline
\hline 
$(R_{7}:(R_{5},4R_{3}))$ & $(R_{7}:2R_{5},3R_{3})$ & $(R_{7}:3R_{5})$\tabularnewline
\hline 
$(R_{7}:2(R_{5},R_{3}),R_{3})$ & $(R_{7}:(R_{5},2R_{3}),4R_{3})$ & $(R_{7}:(R_{5},4R_{3}),2R_{3})$\tabularnewline
\hline 
\multicolumn{3}{|c|}{$\pm\left(\zeta_{5}+\zeta_{5}^{2}+\zeta_{5}^{3}+\zeta_{5}^{4}+\sum_{i=1}^{6}(\zeta_{3}+\zeta_{3}^{2})\zeta_{7}^{i}\right)$}\tabularnewline
\hline 
\multicolumn{3}{|c|}{$\pm\left(\zeta_{3}+\zeta_{3}^{2}+\zeta_{10}+\zeta_{10}^{9}+\sum_{i=1}^{6}(\zeta_{5}+\zeta_{5}^{4})\zeta_{7}^{i}\right)$}\tabularnewline
\hline 
\multicolumn{3}{|c|}{$\pm\left((\zeta_{10}+\zeta_{10}^{9}-1)(\zeta_{7}+\zeta_{7}^{6})+\sum_{i\neq1,6}(\zeta_{5}+\zeta_{5}^{4})\zeta_{7}^{i}\right)$}\tabularnewline
\hline 
\end{longtable}
\end{lem}

\begin{proof}
If $w(G_{0})=15$ or $w(G_{0})\geq17$, then the even part of $G$
must be empty, $2R_{2}$, or $2R_{3}$. If $G$ is of the form (\ref{equation7}),
we can apply Lemma \ref{lemma10}, \ref{lemma12}, and \ref{lemma13}.
The last three entries can be obtained by the same reasoning of (\ref{equation8}),
(\ref{equation9}), and (\ref{equation10}).
\end{proof}
Applying Algorithm \ref{algorithm27} to these entries, the results
confirm Lemma \ref{lemma29}.

\addtocounter{table}{-1}
\begin{lem}
\label{lemma34}Suppose $G=0$ is a relation with $w(G)=24$ and $p=5,7$.
If the even part of $G$ is not empty, $2R_{2}$, or $2R_{3}$, and
$G$ cannot be decomposed as $G=G_{0}+n_{2}R_{2}+n_{3}R_{3}+n_{5}R_{5}$,
where $G_{0}$ is minimal and of weight at least $11$, then $G$
must be one of the following:

\begin{longtable}{|c|c|c|}
\hline 
$R_{5}+R_{3}+8R_{2}$ & $R_{5}+3R_{3}+5R_{2}$ & $R_{5}+5R_{3}+2R_{2}$\tabularnewline
\hline 
$2R_{5}+7R_{2}$ & $2R_{5}+2R_{3}+4R_{2}$ & $2R_{5}+4R_{3}+R_{2}$\tabularnewline
\hline 
$3R_{5}+R_{3}+3R_{2}$ & $3R_{5}+3R_{3}$ & $4R_{5}+2R_{2}$\tabularnewline
\hline 
$(R_{5}:R_{3})+9R_{2}$ & $(R_{5}:R_{3})+2R_{3}+6R_{2}$ & $(R_{5}:R_{3})+4R_{3}+3R_{2}$\tabularnewline
\hline 
$(R_{5}:R_{3})+6R_{3}$ & $(R_{5}:R_{3})+R_{5}+R_{3}+5R_{2}$ & $(R_{5}:R_{3})+R_{5}+3R_{3}+2R_{2}$\tabularnewline
\hline 
$(R_{5}:R_{3})+2R_{5}+4R_{2}$ & $(R_{5}:R_{3})+2R_{5}+2R_{3}+R_{2}$ & $(R_{5}:R_{3})+3R_{5}+R_{3}$\tabularnewline
\hline 
$2(R_{5}:R_{3})+6R_{2}$ & $2(R_{5}:R_{3})+2R_{3}+3R_{2}$ & $2(R_{5}:R_{3})+4R_{3}$\tabularnewline
\hline 
$2(R_{5}:R_{3})+R_{5}+R_{3}+2R_{2}$ & $2(R_{5}:R_{3})+2R_{5}+R_{2}$ & $3(R_{5}:R_{3})+3R_{2}$\tabularnewline
\hline 
$3(R_{5}:R_{3})+2R_{3}$ & $4(R_{5}:R_{3})$ & $(R_{5}:2R_{3})+R_{3}+7R_{2}$\tabularnewline
\hline 
$(R_{5}:2R_{3})+3R_{3}+4R_{2}$ & $(R_{5}:2R_{3})+5R_{3}+R_{2}$ & $(R_{5}:2R_{3})+R_{5}+6R_{2}$\tabularnewline
\hline 
$(R_{5}:2R_{3})+R_{5}+2R_{3}+3R_{2}$ & $(R_{5}:2R_{3})+R_{5}+4R_{3}$ & $(R_{5}:2R_{3})+2R_{5}+R_{3}+2R_{2}$\tabularnewline
\hline 
$(R_{5}:2R_{3})+3R_{5}+R_{2}$ & $(R_{5}:2R_{3})+(R_{5}:R_{3})+R_{3}+4R_{2}$ & $(R_{5}:2R_{3})+2(R_{5}:R_{3})+R_{3}+R_{2}$\tabularnewline
\hline 
$(R_{5}:2R_{3})+2(R_{5}:R_{3})+R_{5}$ & $2(R_{5}:2R_{3})+5R_{2}$ & $2(R_{5}:2R_{3})+2R_{3}+2R_{2}$\tabularnewline
\hline 
$2(R_{5}:2R_{3})+R_{5}+R_{3}+R_{2}$ & $2(R_{5}:2R_{3})+2R_{5}$ & $2(R_{5}:2R_{3})+(R_{5}:R_{3})+2R_{2}$\tabularnewline
\hline 
$3(R_{5}:2R_{3})+R_{3}$ & $(R_{5}:3R_{3})+8R_{2}$ & $(R_{5}:3R_{3})+2R_{3}+5R_{2}$\tabularnewline
\hline 
$(R_{5}:3R_{3})+4R_{3}+2R_{2}$ & $(R_{5}:3R_{3})+R_{5}+R_{3}+4R_{2}$ & $(R_{5}:3R_{3})+2R_{5}+3R_{2}$\tabularnewline
\hline 
$(R_{5}:3R_{3})+2R_{5}+2R_{3}$ & $(R_{5}:3R_{3})+(R_{5}:R_{3})+5R_{2}$ & $(R_{5}:3R_{3})+(R_{5}:R_{3})+2R_{3}+2R_{2}$\tabularnewline
\hline 
$(R_{5}:3R_{3})+(R_{5}:R_{3})+2R_{5}$ & $(R_{5}:3R_{3})+2(R_{5}:R_{3})+2R_{2}$ & $(R_{5}:3R_{3})+2(R_{5}:2R_{3})+R_{2}$\tabularnewline
\hline 
$2(R_{5}:3R_{3})+4R_{2}$ & $2(R_{5}:3R_{3})+2R_{3}+R_{2}$ & $2(R_{5}:3R_{3})+R_{5}+R_{3}$\tabularnewline
\hline 
$2(R_{5}:3R_{3})+(R_{5}:R_{3})+R_{2}$ & $3(R_{5}:3R_{3})$ & $(R_{5}:4R_{3})+R_{3}+6R_{2}$\tabularnewline
\hline 
$(R_{5}:4R_{3})+3R_{3}+3R_{2}$ & $(R_{5}:4R_{3})+5R_{3}$ & $(R_{5}:4R_{3})+R_{5}+5R_{2}$\tabularnewline
\hline 
$(R_{5}:4R_{3})+R_{5}+2R_{3}+2R_{2}$ & $(R_{5}:4R_{3})+2R_{5}+R_{3}+R_{2}$ & $(R_{5}:4R_{3})+3R_{5}$\tabularnewline
\hline 
$(R_{5}:4R_{3})+2(R_{5}:R_{3})+R_{3}$ & $(R_{5}:4R_{3})+(R_{5}:2R_{3})+4R_{2}$ & $2(R_{5}:4R_{3})+3R_{2}$\tabularnewline
\hline 
$2(R_{5}:4R_{3})+2R_{3}$ & $2(R_{5}:4R_{3})+(R_{5}:R_{3})$ & \tabularnewline
\hline 
$R_{7}+R_{3}+7R_{2}$ & $R_{7}+3R_{3}+4R_{2}$ & $R_{7}+5R_{3}+R_{2}$\tabularnewline
\hline 
$R_{7}+R_{5}+6R_{2}$ & $R_{7}+R_{5}+2R_{3}+3R_{2}$ & $R_{7}+R_{5}+4R_{3}$\tabularnewline
\hline 
$R_{7}+2R_{5}+R_{3}+2R_{2}$ & $R_{7}+3R_{5}+R_{2}$ & $R_{7}+(R_{5}:R_{3})+R_{3}+4R_{2}$\tabularnewline
\hline 
$R_{7}+2(R_{5}:R_{3})+R_{3}+R_{2}$ & $R_{7}+2(R_{5}:R_{3})+R_{5}$ & $R_{7}+(R_{5}:2R_{3})+5R_{2}$\tabularnewline
\hline 
$2R_{7}+5R_{2}$ & $R_{7}+(R_{5}:2R_{3})+2R_{3}+2R_{2}$ & $2R_{7}+2R_{3}+2R_{2}$\tabularnewline
\hline 
$2R_{7}+R_{5}+R_{3}+R_{2}$ & $R_{7}+(R_{5}:2R_{3})+2R_{5}$ & $2R_{7}+2R_{5}$\tabularnewline
\hline 
$2R_{7}+(R_{5}:R_{3})+2R_{2}$ & $R_{7}+2(R_{5}:2R_{3})+R_{3}$ & $2R_{7}+(R_{5}:2R_{3})+R_{3}$\tabularnewline
\hline 
$3R_{7}+R_{3}$ & $(R_{7}:R_{3})+8R_{2}$ & $(R_{7}:R_{3})+2R_{3}+5R_{2}$\tabularnewline
\hline 
$(R_{7}:R_{3})+4R_{3}+2R_{2}$ & $(R_{7}:R_{3})+R_{5}+R_{3}+4R_{2}$ & $(R_{7}:R_{3})+2R_{5}+3R_{2}$\tabularnewline
\hline 
$(R_{7}:R_{3})+2R_{5}+2R_{3}$ & $(R_{7}:R_{3})+(R_{5}:R_{3})+5R_{2}$ & $(R_{7}:R_{3})+(R_{5}:R_{3})+2R_{3}+2R_{2}$\tabularnewline
\hline 
$(R_{7}:R_{3})+(R_{5}:R_{3})+2R_{5}$ & $(R_{7}:R_{3})+2(R_{5}:R_{3})+2R_{2}$ & $(R_{5}:3R_{3})+2R_{7}+R_{2}$\tabularnewline
\hline 
$(R_{7}:R_{3})+2(R_{5}:2R_{3})+R_{2}$ & $(R_{7}:R_{3})+2R_{7}+R_{2}$ & $(R_{7}:R_{3})+(R_{5}:3R_{3})+4R_{2}$\tabularnewline
\hline 
$2(R_{7}:R_{3})+4R_{2}$ & $2(R_{7}:R_{3})+2R_{3}+R_{2}$ & $2(R_{7}:R_{3})+R_{5}+R_{3}$\tabularnewline
\hline 
$2(R_{7}:R_{3})+(R_{5}:R_{3})+R_{2}$ & $(R_{7}:R_{3})+2(R_{5}:3R_{3})$ & $2(R_{7}:R_{3})+(R_{5}:3R_{3})$\tabularnewline
\hline 
$3(R_{7}:R_{3})$ & $(R_{7}:2R_{3})+R_{3}+6R_{2}$ & $(R_{7}:2R_{3})+3R_{3}+3R_{2}$\tabularnewline
\hline 
$(R_{7}:2R_{3})+5R_{3}$ & $(R_{7}:2R_{3})+R_{5}+5R_{2}$ & $(R_{7}:2R_{3})+R_{5}+2R_{3}+2R_{2}$\tabularnewline
\hline 
$(R_{7}:2R_{3})+2R_{5}+R_{3}+R_{2}$ & $(R_{7}:2R_{3})+3R_{5}$ & $(R_{7}:2R_{3})+2(R_{5}:R_{3})+R_{3}$\tabularnewline
\hline 
$(R_{5}:4R_{3})+R_{7}+4R_{2}$ & $(R_{7}:2R_{3})+(R_{5}:2R_{3})+4R_{2}$ & $(R_{7}:2R_{3})+R_{7}+4R_{2}$\tabularnewline
\hline 
$2(R_{7}:2R_{3})+3R_{2}$ & $2(R_{7}:2R_{3})+2R_{3}$ & $2(R_{7}:2R_{3})+(R_{5}:R_{3})$\tabularnewline
\hline 
$(R_{7}:3R_{3})+7R_{2}$ & $(R_{7}:R_{5})+7R_{2}$ & $(R_{7}:3R_{3})+2R_{3}+4R_{2}$\tabularnewline
\hline 
$(R_{7}:R_{5})+2R_{3}+4R_{2}$ & $(R_{7}:3R_{3})+4R_{3}+R_{2}$ & $(R_{7}:R_{5})+4R_{3}+R_{2}$\tabularnewline
\hline 
$(R_{7}:3R_{3})+2R_{5}+2R_{2}$ & $(R_{7}:R_{5})+2R_{5}+2R_{2}$ & $(R_{7}:3R_{3})+(R_{5}:R_{3})+4R_{2}$\tabularnewline
\hline 
$(R_{7}:R_{5})+(R_{5}:R_{3})+4R_{2}$ & $(R_{7}:3R_{3})+2(R_{5}:R_{3})+R_{2}$ & $(R_{7}:R_{5})+2(R_{5}:R_{3})+R_{2}$\tabularnewline
\hline 
$(R_{7}:3R_{3})+2(R_{5}:2R_{3})$ & $(R_{7}:3R_{3})+2R_{7}$ & $(R_{7}:R_{5})+2(R_{5}:2R_{3})$\tabularnewline
\hline 
$(R_{7}:R_{5})+2R_{7}$ & $2(R_{7}:3R_{3})+2R_{2}$ & $2(R_{7}:R_{5})+2R_{2}$\tabularnewline
\hline 
$2(R_{7}:4R_{3})+R_{2}$ & $2(R_{7}:R_{5},R_{3})+R_{2}$ & $2(R_{7}:(R_{5}:R_{3}))+R_{2}$\tabularnewline
\hline 
$(R_{7}:5R_{3})+2(R_{5}:R_{3})$ & $(R_{7}:R_{5},2R_{3})+2(R_{5}:R_{3})$ & $(R_{7}:(R_{5},R_{3}),R_{3})+2(R_{5}:R_{3})$\tabularnewline
\hline 
$(R_{7}:(R_{5},2R_{3}))+2(R_{5}:R_{3})$ & $2(R_{7}:5R_{3})$ & $2(R_{7}:R_{5},2R_{3})$\tabularnewline
\hline 
$2(R_{7}:(R_{5},R_{3}),R_{3})$ & $2(R_{7}:(R_{5},2R_{3}))$ & \tabularnewline
\hline 
\end{longtable}
\end{lem}

Let $G$ be one of these entries. Suppose $G=G_{1}+\cdots+G_{k}+\sum_{q\leq7}n_{q}R_{q}$,
where $G_{i}\neq R_{q}$ are minimal, is stable under complex conjugation.
By Lemma \ref{lemma11}, $G_{1}+\cdots+G_{k}$ is also stable under
complex conjugation.

Suppose $k=2$. If $w(G_{1})+w(G_{2})\geq18$, and $G_{1}$ and $G_{2}$
are not conjugate to each other, then by the same reasoning of Lemma
\ref{lemma32} and \ref{lemma33}, all the terms of $G$ belong to
$\mu_{420}$. Otherwise,
\begin{lem}
\label{lemma35}If $G_{0}=G_{1}+G_{2}$, where $G_{1},G_{2}\neq R_{q}$
are minimal, and $w(G_{0})=12$, $13$, $14$, or $16$, is stable
under complex conjugation, then $G_{0}$ satisfies Condition \ref{condition19}.
\end{lem}

\begin{proof}
Let $G_{0}=\xi_{1}(R_{p_{1}}:\cdots)+\xi_{2}(R_{p_{2}}:\cdots)$.
If $p_{1}\neq p_{2}$, then the conclusion follows from the same reasoning
of Lemma \ref{lemma21}. Otherwise, $G_{0}$ must be one of
\begin{align*}
 & 2(G_{5}:R_{3}),(G_{5}:R_{3})+(G_{5}:2R_{3}),(G_{5}:R_{3})+(G_{5}:3R_{3}),\\
 & 2(G_{5}:2R_{3}),(G_{5}:2R_{3})+(G_{5}:4R_{3}),2(G_{5}:3R_{3}),2(R_{7}:R_{3}).
\end{align*}
We can get the conclusion by verifying every $\xi_{1},\xi_{2}\in\mu_{60}$
or $\mu_{84}$.
\end{proof}
If $k=3$, then $w(G_{1})+w(G_{2})+w(G_{3})\geq18$. If any two of
$G_{i}$ are not conjugate to each other, then by the same reasoning
of Lemma \ref{lemma32} and \ref{lemma33}, all the terms of $G$
belong to $\mu_{420}$.

For $k=4$, we have the following lemma.
\begin{lem}
\label{lemma36}If $G=\xi_{1}(G_{5}:R_{3})+\xi_{2}(G_{5}:R_{3})+\xi_{3}(G_{5}:R_{3})+\xi_{4}(G_{5}:R_{3})$
is stable under complex conjugation, then either $G$ satisfies Condition
\ref{condition19} or $\xi_{i}\in\mu_{60}$.
\end{lem}

\begin{proof}
Consider the graph $\mathfrak{g}$ (in the sense of graph theory)
with vertices $\xi_{i}$. If some term of $\xi_{i}(G_{5}:R_{3})$
and some term of $\xi_{j}(G_{5}:R_{3})$ are conjugate to each other,
we draw an edge $\{\xi_{i},\xi_{j}\}$ between $\xi_{i}$ and $\xi_{j}$,
where $i=j$ is allowed. If $\{\xi_{i},\xi_{j}\}\in\mathfrak{g}$,
then $\xi_{i}\xi_{j}\in\mu_{30}$.

If there is only one edge at $\xi_{i}$, then either $\xi_{i}(G_{5}:R_{3})$
is itself stable under complex conjugation, or can be paired with
another $\xi_{j}(G_{5}:R_{3})$ which is its complex conjugate. The
conclusion then follows from Lemma \ref{lemma26} and \ref{lemma35}.

Now we assume that there are at least two edges at each vertex. Suppose
the conclusion is false by assuming $\xi_{1}\notin\mu_{60}$, then
$\{\xi_{1},\xi_{1}\}\notin\mathfrak{g}$. Without loss of generality,
assume $\{\xi_{1},\xi_{2}\},\{\xi_{1},\xi_{3}\}\in\mathfrak{g}$.
If $\{\xi_{2},\xi_{2}\}$, $\{\xi_{2},\xi_{3}\}$, or $\{\xi_{3},\xi_{3}\}\in\mathfrak{g}$,
then by the same reasoning of Lemma \ref{lemma26}, we have $\xi_{1}\in\mu_{60}$,
which is a contradiction. Therefore, $\{\xi_{2},\xi_{4}\},\{\xi_{3},\xi_{4}\}\in\mathfrak{g}$,
and by the same reasoning, $\{\xi_{1},\xi_{4}\},\{\xi_{4},\xi_{4}\}\notin\mathfrak{g}$.
In summary,
\[
\mathfrak{g}=\{\xi_{1},\xi_{2},\xi_{3},\xi_{4},\{\xi_{1},\xi_{2}\},\{\xi_{1},\xi_{3}\},\{\xi_{2},\xi_{4}\},\{\xi_{3},\xi_{4}\}\}.
\]
Thus $\xi_{2}(G_{5}:R_{3})+\xi_{3}(G_{5}:R_{3})$ can also be decomposed
as $\overline{\xi_{1}}(G_{5}:R_{3})+\overline{\xi_{4}}(G_{5}:R_{3})$.
\end{proof}
We have shown that all the entries listed in Lemma \ref{lemma34}
either satisfy Condition \ref{condition19} or confirm Lemma \ref{lemma29}.
Now suppose $G=G_{0}+2G_{1}+\cdots+2G_{k}$ is of the form (\ref{equation11}),
and define $D$, $d$, $n$, and $m$ as in Algorithm \ref{algorithm28}.
For a general $G$, we can apply Algorithm \ref{algorithm28} by choosing
$m$ such that for any $d\leq m$, $nd$ divides some element of $N$.
However, for some cases, it will be much more efficient to apply Algorithm
\ref{algorithm23} and \ref{algorithm27}. Here are some examples.

If the even part of $G$ is $4R_{2}$, $2R_{2}+2R_{3}$, $2R_{5}$,
$6R_{2}$, or $4R_{3}$, we can always apply Algorithm \ref{algorithm27},
provided $U_{24}\backslash G_{0}$ has at most one free variable.

If $G=G_{0}+2G_{1}$, then $d\leq4$ or $d=\infty$. We can always
try to
\begin{itemize}
\item Apply Algorithm \ref{algorithm23} by taking $q=5$ or $7$ to conclude
that (\ref{equation1}) cannot be of the form $G$,
\item Apply Algorithm \ref{algorithm23} by taking $q=5$ or $7$ to conclude
that $d=\infty$ cannot happen,
\item Apply Algorithm \ref{algorithm27} to conclude that $d=\infty$ cannot
happen.
\end{itemize}
If $G=2G_{1}+2R_{3}$ or $2G_{1}+2R_{2}$, then it turns out that
the diagonal of $D$ must be $(1,0)$, $(1,2)$, $(1,4)$, or $(1,6)$.
The last three cases confirm Lemma \ref{lemma29}. We can apply Algorithm
\ref{algorithm27} to conclude that the first case cannot happen.

After applying a combination of Algorithm \ref{algorithm23}, \ref{algorithm27},
and \ref{algorithm28} to all the entries listed in Lemma \ref{lemma34},
Lemma \ref{lemma29} can be fully confirmed.

\subsection{\label{section3.6}The application of Theorem \texorpdfstring{\ref{theorem15}}{15}}

\addtocounter{equation}{-1}

Suppose $[s_{1},s_{2},s_{3},t_{1},t_{2},t_{3}]\in\mathcal{S}$ such
that $s_{i},t_{i}\in\mu_{n}$. Let
\[
\sigma_{i}=\frac{n}{2\pi}\textup{Arg}(s_{i})\text{ and }\tau_{i}=\frac{n}{2\pi}\textup{Arg}(t_{i}),
\]
then the identity (\ref{equation4}) implies
\begin{equation}
\frac{(s_{3}/s_{1}-1)(s_{2}-1)}{(s_{2}/s_{1}-1)(s_{3}-1)}=\pm\frac{(t_{3}/t_{1}-1)(t_{2}-1)}{(t_{2}/t_{1}-1)(t_{3}-1)}\Leftrightarrow\frac{e_{\sigma_{3}-\sigma_{1}}e_{\sigma_{2}}}{e_{\sigma_{2}-\sigma_{1}}e_{\sigma_{3}}}=\frac{e_{\tau_{3}-\tau_{1}}e_{\tau_{2}}}{e_{\tau_{2}-\tau_{1}}e_{\tau_{3}}}.\tag{4\ensuremath{'}}\label{equation4'}
\end{equation}

\addtocounter{equation}{1}
\begin{lyxalgorithm}
\label{algorithm37}Given a positive integer $n$, let $\widehat{E}_{n}$
be the $\mathbb{Q}$-vector space generated by the free symbols $\{\widehat{e}_{k}:1\leq k<n\}$,
$\widehat{F}_{n}$ the subspace generated by the relations (\ref{equation5})
and (\ref{equation6}) (where $e_{k}$ are replaced with $\widehat{e}_{k}$),
$\widehat{F}_{n}^{\perp}$ the orthogonal complement of $\widehat{F}_{n}$,
$\{\widehat{f}_{1},\cdots,\widehat{f}_{j}\}$ a basis of $\widehat{F}_{n}^{\perp}$.
For any $1\leq\sigma_{1},\sigma_{2},\sigma_{3}<n$ such that any two
of them are not equal, we calculate the inner products
\[
\widehat{g}_{\sigma_{1},\sigma_{2},\sigma_{3}}:=\left(\frac{\widehat{e}_{\sigma_{3}-\sigma_{1}}\widehat{e}_{\sigma_{2}}}{\widehat{e}_{\sigma_{2}-\sigma_{1}}\widehat{e}_{\sigma_{3}}}\cdot\widehat{f}_{1},\cdots,\frac{\widehat{e}_{\sigma_{3}-\sigma_{1}}\widehat{e}_{\sigma_{2}}}{\widehat{e}_{\sigma_{2}-\sigma_{1}}\widehat{e}_{\sigma_{3}}}\cdot\widehat{f}_{j}\right).
\]
By Theorem \ref{theorem15},
\[
\text{(\ref{equation4'})}\Leftrightarrow\left.\frac{\widehat{e}_{\sigma_{3}-\sigma_{1}}\widehat{e}_{\sigma_{2}}}{\widehat{e}_{\sigma_{2}-\sigma_{1}}\widehat{e}_{\sigma_{3}}}\middle/\frac{\widehat{e}_{\tau_{3}-\tau_{1}}\widehat{e}_{\tau_{2}}}{\widehat{e}_{\tau_{2}-\tau_{1}}\widehat{e}_{\tau_{3}}}\right.\in\widehat{F}_{n}\Leftrightarrow\widehat{g}_{\sigma_{1},\sigma_{2},\sigma_{3}}=\widehat{g}_{\tau_{1},\tau_{2},\tau_{3}}.
\]
\begin{itemize}
\item We first find all $(\sigma_{1},\sigma_{2},\sigma_{3})$ and $(\tau_{1},\tau_{2},\tau_{3})$
satisfying (\ref{equation4'}),
\item and then among them, find all $(s_{1},s_{2},s_{3})$ and $(t_{1},t_{2},t_{3})$
satisfying (\ref{equation4}),
\item and finally among them, find all solutions such that their associated
relations (\ref{equation1}) cannot be decomposed as $n_{2}R_{2}+n_{3}R_{3}$
for some $n_{2}$ and $n_{3}$.
\end{itemize}
\end{lyxalgorithm}

Applying Algorithm \ref{algorithm37} to all $n\in N$ given in Lemma
\ref{lemma29}, we will get $576$ zero-parameter solutions.

\subsection{\label{section3.7}Comparisons of Algorithm \texorpdfstring{\ref{algorithm23},
\ref{algorithm27}, \ref{algorithm28}, and \ref{algorithm37}}{23, 27, 28, and 37}}

In summary, we obtain the complete list of $\mathcal{S}$ by finding
all the solutions of (\ref{equation1}) and (\ref{equation2}) such
that (\ref{equation1}) is of the form $G$. Depending on the largest
prime $p$ involved in $G$, we use the following strategies:
\begin{itemize}
\item $p=17,19,23$: We first give a classification for all such $G$, and
then apply Algorithm \ref{algorithm23} to prove that (\ref{equation1})
cannot be any of them.
\item $p=11,13$: We first apply Algorithm \ref{algorithm23} to give some
necessary conditions that (\ref{equation1}) must satisfy. We then
eliminate those unqualified $G$ and give a classification for the
rest. For those survived $G$, we first try to apply Algorithm \ref{algorithm23}
to prove that (\ref{equation1}) cannot be of that form. If Algorithm
\ref{algorithm23} fails, we apply Algorithm \ref{algorithm27} to
find the solutions.
\item $p=2,3$: We first give a classification for all such $G$, and then
apply Algorithm \ref{algorithm28} to find the solutions.
\item $p=5,7$: We first prove that Lemma \ref{lemma29} holds unless each
minimal relation occurring in $G$ has a moderate weight, for which
we are able to give a classification. We then apply Algorithm \ref{algorithm23},
\ref{algorithm27}, and \ref{algorithm28} together to show that Lemma
\ref{lemma29} holds in general. For each element given in Lemma \ref{lemma29},
we apply Algorithm \ref{algorithm37} to find the solutions.
\end{itemize}
We see that the first step for each case is classification. On the
one hand, there are only a few minimal relations with $p\leq5$, so
it is easy to classify all $G$ with $p\leq5$. On the other hand,
when $p\geq11$, $w(G)$ is smaller or slightly larger than $2p$,
so that we can apply Lemma \ref{lemma10} to classify all $G$ with
$p\geq11$. However, for $p=7$, we are unable to give a full description
for those $G$ containing a long minimal relation. To tackle this
issue, we first prove Lemma \ref{lemma29} and then apply Algorithm
\ref{algorithm37}.

Although Algorithm \ref{algorithm37} is powerful, it fails to give
those positive-parameter solutions since their orders can be arbitrarily
large. This is why we need Algorithm \ref{algorithm28}.

Now suppose $G=G_{0}+2G_{1}+\cdots+2G_{k}$ is of the form (\ref{equation11}),
and define $d$ and $n$ as in Algorithm \ref{algorithm28}. If $w(G_{i})$
is large for some $i$, then Algorithm \ref{algorithm28} will be
very inefficient. If $d=\infty$, then Algorithm \ref{algorithm37}
fails. If $nd\neq\infty$ is large, then Algorithm \ref{algorithm37}
will be very inefficient. This is why we need Algorithm \ref{algorithm23}
and \ref{algorithm27}.

\section{\label{section4}Proof of Theorem \texorpdfstring{\ref{theorem3}}{3}}

Recall that
\begin{eqnarray*}
 &  & (0,s_{1},\cdots,s_{n-1})\sim(0,t_{1},\cdots,t_{n-1})\text{ in }\mathcal{C}_{n}^{+}\\
 & \Leftrightarrow & Z_{i}:=[s_{1},s_{2},s_{i},t_{1},t_{2},t_{i}]\in\mathcal{S}^{+}\text{ for any }3\leq i\leq n-1.
\end{eqnarray*}
Now we will find the longest nontrivial projectively equivalent pairs
in four steps. For $k=3$, $2$, $1$, and $0$, we assume that all
$Z_{i}$ are at least type-$k$ and determine how large $n$ can be,
where an element of $\mathcal{S}^{+}$ is said to be type-$k$ if
it is equivalent to some $k$-parameter solution.

The $k=3$ case, which can be easily done by hand, serves as a prototype
for the following two algorithms.

\addtocounter{table}{-1}
\begin{lem}
\label{lemma38}If all $Z_{i}$ are type-three, then $n\leq5$.
\end{lem}

\begin{proof}
We choose $Z_{3}$ and $Z_{4}$ from the rows of

\begin{longtable}{|c|c|c|c|c|c|}
\hline 
$x$ & $y$ & $z_{1}$ & $x$ & $x-z_{1}$ & $x-y$\tabularnewline
\hline 
$x$ & $y$ & $z_{2}$ & $y-z_{2}$ & $y$ & $y-x$\tabularnewline
\hline 
$x$ & $y$ & $z_{3}$ & $z_{3}-y$ & $z_{3}-x$ & $z_{3}$\tabularnewline
\hline 
$x$ & $y$ & $z_{4}$ & $-x$ & $z_{4}-x$ & $y-x$\tabularnewline
\hline 
$x$ & $y$ & $z_{5}$ & $z_{5}-y$ & $-y$ & $x-y$\tabularnewline
\hline 
$x$ & $y$ & $z_{6}$ & $y-z_{6}$ & $x-z_{6}$ & $-z_{6}$\tabularnewline
\hline 
\end{longtable}

and assume that the resulting pair $(0,s_{1},s_{2},s_{3},s_{4})\sim(0,t_{1},t_{2},t_{3},t_{4})$
is nontrivial. This is only possible when we take the $j$-th and
$(j+3)$-th rows for some $1\leq j\leq3$.
\begin{itemize}
\item If $j=1$, we have $x=-x$ and $x-z_{1}=z_{4}-x$, then $x=1/2$ and
$z_{4}=-z_{1}$. The resulting pair is
\[
(0,1/2,y,z_{1},-z_{1})\sim(0,1/2,1/2-z_{1},1/2-y,1/2+y).
\]
\item If $j=2$, we have $y-z_{2}=z_{5}-y$ and $y=-y$, then $y=1/2$ and
$z_{5}=-z_{2}$. The resulting pair is
\[
(0,x,1/2,z_{2},-z_{2})\sim(0,1/2-z_{2},1/2,1/2-x,1/2+x).
\]
\item If $j=3$, we have $z_{3}-y=y-z_{6}$ and $z_{3}-x=x-z_{6}$, then
$y=1/2+x$ and $z_{6}=2x-z_{3}$. The resulting pair is
\[
(0,x,1/2+x,z_{3},2x-z_{3})\sim(0,1/2-x+z_{3},z_{3}-x,z_{3},z_{3}-2x).
\]
\end{itemize}
Since $\{1,4\}$, $\{2,5\}$, and $\{3,6\}$ are mutually disjoint,
we have $n\leq5$.
\end{proof}
\begin{lyxalgorithm}
\label{algorithm39}Suppose all $Z_{i}$ are at least type-two, and
at least one $Z_{i}$ is type-two. We will find the largest $n$ inductively.
Let $[\widehat{s}_{1},\widehat{s}_{2},\widehat{s}_{3},\widehat{t}_{1},\widehat{t}_{2},\widehat{t}_{3}]$
be a general type-two solution (with two free variables). For any
\[
(0,s_{1},\cdots,s_{m})\sim(0,t_{1},\cdots,t_{m})
\]
that we already have, we solve the equation
\begin{equation}
(s_{1},s_{2},t_{1},t_{2})=(\widehat{s}_{1},\widehat{s}_{2},\widehat{t}_{1},\widehat{t}_{2}).\label{equation12}
\end{equation}
\begin{itemize}
\item If the resulting $s_{1},s_{2},t_{1},t_{2}$ are constants, we add
their lowest common denominator $d$ into a set $D$.
\item If the resulting $s_{1},s_{2},t_{1},t_{2}$ have free variables, we
extend the previous pair to
\[
(0,s_{1},\cdots,s_{m},s_{m+1})=(0,s_{1},\cdots,s_{m},\widehat{s}_{3})\sim(0,t_{1},\cdots,t_{m},\widehat{t}_{3})=(0,t_{1},\cdots,t_{m},t_{m+1}).
\]
\end{itemize}
When this process terminates, we get a set $D$ and some long projectively
equivalent pairs with free variables. Finally, in order to get those
long projectively equivalent pairs without free variables, we find
all $Z_{i}$ such that $s_{1},s_{2},t_{1},t_{2}\in(1/d)\mathbb{Z}/\mathbb{Z}$
for some $d\in D$. The conclusion is: if all $Z_{i}$ are at least
type-two, then $n\leq10$.
\end{lyxalgorithm}

\begin{lyxalgorithm}
\label{algorithm40}Suppose all $Z_{i}$ are at least type-one, and
at least one $Z_{i}$ is type-one. Let $\mathcal{S}_{k}^{+}$ be the
collection of general type-$k$ solutions (with $k$ free variables).
For any
\[
[\widehat{s}_{1},\widehat{s}_{2},\widehat{s}_{3},\widehat{t}_{1},\widehat{t}_{2},\widehat{t}_{3}]\in\mathcal{S}_{1}^{+}\text{ and }[s_{1},s_{2},s_{3},t_{1},t_{2},t_{3}]\in\mathcal{S}_{2}^{+}\cup\mathcal{S}_{3}^{+},
\]
we try to solve the equation (\ref{equation12}). However, due to
the large size of $\mathcal{S}_{1}^{+}$, it is better to solve (\ref{equation12})
only partially. (We have used the same trick in Algorithm \ref{algorithm28}.)
\begin{itemize}
\item If we know that the resulting $s_{1},s_{2},t_{1},t_{2}\in(1/d)\mathbb{Z}/\mathbb{Z}$
are constants, we add $d$, which may not be their lowest common denominator,
into a set $D_{1}$.
\item If the resulting $[s_{1},s_{2},s_{3},t_{1},t_{2},t_{3}]$ has one
free variable, we add it into a set $\mathcal{S}_{1,new}^{+}$.
\end{itemize}
When this process terminates, we get two sets $D_{1}$ and $\mathcal{S}_{1,new}^{+}$.
For any
\[
[s_{1},s_{2},s_{3},t_{1},t_{2},t_{3}],[\widehat{s}_{1},\widehat{s}_{2},\widehat{s}_{3},\widehat{t}_{1},\widehat{t}_{2},\widehat{t}_{3}]\in\mathcal{S}_{1}^{+}\cup\mathcal{S}_{1,new}^{+},
\]
we try to solve the equation (\ref{equation12}) again. Since both
of them have exactly one free variable, (\ref{equation12}) becomes
\[
V_{0}+V_{1}x=(s_{1},s_{2},t_{1},t_{2})=(\widehat{s}_{1},\widehat{s}_{2},\widehat{t}_{1},\widehat{t}_{2})=\widehat{V}_{0}+\widehat{V}_{1}\widehat{x},
\]
where $V_{0},\widehat{V}_{0}\in(\mathbb{Q}/\mathbb{Z})^{4}$ and $V_{1},\widehat{V}_{1}\in\mathbb{Z}^{4}$.
\begin{itemize}
\item If we know that the resulting $s_{1},s_{2},t_{1},t_{2}\in(1/d)\mathbb{Z}/\mathbb{Z}$
are constants, we add $d$, which may not be their lowest common denominator,
into a set $D_{2}$.
\item Long projectively equivalent pairs with free variables possibly exist
only when $V_{1}$ and $\widehat{V}_{1}$ are linearly dependent.
It turns out that the greatest common divisor of $V_{1}$ must be
$1$ or $2$. Moreover,
\[
\begin{cases}
V_{0}\in((1/12)\mathbb{Z}/\mathbb{Z})^{4}, & \text{if }\gcd(V_{1})=1,\\
V_{0}\in((1/6)\mathbb{Z}/\mathbb{Z})^{4}, & \text{if }\gcd(V_{1})=2.
\end{cases}
\]
Therefore, for each element of $\mathcal{S}_{1}^{+}\cup\mathcal{S}_{1,new}^{+}$,
we can substitute its free variable $x$ with some $c_{0}+c_{1}x$,
where $c_{0}\in(1/12)\mathbb{Z}/\mathbb{Z}$ and $c_{1}\in\{\pm1,\pm2\}$.
\end{itemize}
When this process terminates, we get a set $D_{2}$ and some long
projectively equivalent pairs with free variables. Finally, in order
to get those long projectively equivalent pairs without free variables,
we find all $Z_{i}$ such that $s_{1},s_{2},t_{1},t_{2}\in(1/d)\mathbb{Z}/\mathbb{Z}$
for some $d\in D_{1}\cup D_{2}$. The conclusion is: if all $Z_{i}$
are at least type-one, then $n\leq14$. Moreover, if $n\geq11$, then
$s_{1},s_{2},t_{1},t_{2}\in(1/30)\mathbb{Z}/\mathbb{Z}$.
\end{lyxalgorithm}

Therefore, if $n$ is the maximal length, then either $s_{1},s_{2},t_{1},t_{2}\in(1/30)\mathbb{Z}/\mathbb{Z}$
or at least four $Z_{i}$ are type-zero. It turns out that the latter
condition implies $s_{1},s_{2},t_{1},t_{2}\in(1/84)\mathbb{Z}/\mathbb{Z}$
or $(1/120)\mathbb{Z}/\mathbb{Z}$. Once we find all $Z_{i}$ such
that $s_{1},s_{2},t_{1},t_{2}\in(1/84)\mathbb{Z}/\mathbb{Z}$ or $(1/120)\mathbb{Z}/\mathbb{Z}$,
Theorem \ref{theorem3} will be proved.

\begin{longtable}{|>{\centering}p{1.3cm}|>{\centering}p{1.3cm}|>{\centering}p{1.3cm}|>{\centering}p{1.3cm}|>{\centering}p{1.3cm}|>{\centering}p{1.3cm}|>{\centering}p{1.3cm}|}
\caption{The zero-parameter solutions in $\mathcal{S}^{+}/\approx$}
\label{table2}\tabularnewline
\hline 
Order & $s_{1}$ & $s_{2}$ & $s_{3}$ & $t_{1}$ & $t_{2}$ & $t_{3}$\tabularnewline
\hline 
26 & 1/26 & 3/26 & 9/26 & 2/13 & 11/26 & 8/13\tabularnewline
\hline 
30 & 1/30 & 1/15 & 1/10 & 1/30 & 1/10 & 4/15\tabularnewline
\hline 
30 & 1/30 & 1/15 & 1/10 & 1/30 & 1/6 & 11/15\tabularnewline
\hline 
30 & 1/30 & 1/15 & 1/10 & 1/15 & 1/6 & 3/10\tabularnewline
\hline 
30 & 1/30 & 1/15 & 1/10 & 1/15 & 7/30 & 8/15\tabularnewline
\hline 
30 & 1/30 & 1/15 & 1/10 & 1/10 & 7/30 & 11/30\tabularnewline
\hline 
\end{longtable}

\textbf{Acknowledgments.} The author is deeply grateful to Fedor Bogomolov
for suggesting this problem and indicating the crucial Theorem \ref{theorem15}.

\addcontentsline{toc}{chapter}{References}

Address: National Center for Theoretical Sciences, National Taiwan
University, Taipei, Taiwan\\
Email: \href{mailto:fu@ncts.ntu.edu.tw}{fu@ncts.ntu.edu.tw}
\end{document}